\documentclass[12pt]
{amsart}
\usepackage{amssymb}
\usepackage{latexsym}
\usepackage{amsmath}
\newtheorem{theorem}{Theorem}[section]
\newtheorem{lemma}[theorem]{Lemma}
\newtheorem{proposition}[theorem]{Proposition}
\newtheorem{corollary}[theorem]{Corollary}
\newtheorem{definition}[theorem]{Definition}
\newtheorem{remark}[theorem]{Remark}
\numberwithin{equation}{section}

\def\shh{{\mathfrak h}}

   \def\sH{{\mathfrak H}}   
   \def\sK{{\mathfrak K}}   \def\sL{{\mathfrak L}}
\def\sM{{\mathfrak M}}   \def\sN{{\mathfrak N}}   
      
\def\sS{{\mathfrak S}}

      \def\dC{{\mathbb C}}

   \def\dN{{\mathbb N}}   
      \def\dR{{\mathbb R}}

\def\cA{{\mathcal A}}   \def\cB{{\mathcal B}}   \def\cC{{\mathcal C}}
\def\cD{{\mathcal D}}      
   \def\cH{{\mathcal H}}   
   \def\cK{{\mathcal K}}   \def\cL{{\mathcal L}}
\def\cM{{\mathcal M}}   \def\cN{{\mathcal N}}   
\def\cP{{\mathcal P}}   \def\cQ{{\mathcal Q}}   \def\cR{{\mathcal R}}
\def\cS{{\mathcal S}}      
\def\cV{{\mathcal V}}      \def\cX{{\mathcal X}}
\def\cY{{\mathcal Y}}   \def\cZ{{\mathcal Z}}

      \def\bL{{\mathbf L}}

\def\half{{\frac{1}{2}}}
\def\dim{{\rm dim\,}}

\def\RE{{\rm Re\,}}
\def\IM{{\rm Im\,}}
\def\Ext{{\rm Ext\,}}

\def\clos{{\rm clos\,}}
\def\ran{{\rm ran\,}}
\def\cran{{\rm \overline{ran}\,}}
\def\dom{{\rm dom\,}}
\def\mul{{\rm mul\,}}

\def\cdom{{\rm \overline{dom}\,}}
\def\dim{{\rm dim\,}}
\def\graph{{\rm graph\,}}

\def\cspan{{\rm \overline{span}\, }}
\def\uphar{{\upharpoonright\,}}
\def\wt{\widetilde}

\def\f{\varphi}

\begin{document}

\title[Q-functions of nonnegative operators]
{ Q-functions and boundary triplets\\
 of nonnegative operators }

\author[Yury Arlinski\u{\i}]{Yu.M. Arlinski\u{\i}}
\address{Department of Mathematical Analysis  \\
East Ukrainian National University  \\
Kvartal Molodyozhny 20-A  \\
Lugansk 91034  \\
Ukraine} \email{yury.arlinskii@gmail.com}

\author[Seppo Hassi]{S. Hassi}
\address{Department of Mathematics and Statistics  \\
University of Vaasa  \\
P.O. Box 700  \\
65101 Vaasa  \\
Finland} \email{sha@uwasa.fi}

\dedicatory{Dedicated to Lev Aronovich Sakhnovich in the occasion of
his 80-th birthday}

\begin{abstract}Operator-valued $Q$-functions for special pairs of nonnegative
selfadjoint extensions of nonnegative not necessarily densely
defined operators are defined and their analytical properties are
studied. It is shown that the Kre\u\i n-Ovcharenko statement
announced in \cite{KrO2} is valid only for $Q$-functions of densely
defined symmetric operators with finite deficiency indices. A
general class of boundary triplets for a densely defined nonnegative
operator is constructed such that the corresponding Weyl functions
are of Kre\u\i n-Ovcharenko type.

\end{abstract}
\maketitle

\section{Introduction}
{\bf Notations.} We use the symbols $\dom T$, $\ran T$, $\ker T$ for
the domain, the range, and the null-subspace of a linear operator
$T$. The closures of $\dom T$, $\ran T$ are denoted by $\cdom T$,
$\cran T$, respectively. The identity operator in a Hilbert space
$\sH$ is denoted by  $I$ and sometimes by $I_\sH$. If $\sL$ is a
subspace, i.e., a closed linear subset of $\sH$, the orthogonal
projection in $\sH$ onto $\sL$ is denoted by $P_\sL.$ The notation
$T\uphar \cN$ means the restriction of a linear operator $T$ on the
set $\cN\subset\dom T$. The resolvent set of $T$ is denoted by
$\rho(T)$. The linear space of bounded operators acting between
Hilbert spaces $\sH$ and $\sK$ is denoted by $\bL(\sH,\sK)$ and the
Banach algebra $\bL(\sH,\sH)$ by $\bL(\sH).$ A linear operator $\cA$
in a Hilbert space is called nonnegative if $(\cA f,f)\ge 0$ for all
$f\in \dom \cA.$

\vskip 0.3cm Let $\sH$ be a separable complex Hilbert space and let
$S$ be a closed symmetric operator with equal deficiency indices in
$\sH$. We do not suppose that $S$ is densely defined. As it is well
known the Kre\u\i n's resolvent formula for canonical and
generalized resolvents plays crucial role in the spectral theory of
selfadjoint extensions and its numerous applications. The essential
part of this formula is the $Q$-function of $S$.
Denote by $\sN_z$ the defect subspace of $S$, i.e.,
\[
\sN_z=\sH\ominus \ran(S-\bar zI).
\]
or, equivalently, $\sN_z=\ker(S^*-zI).$ Choose a selfadjoint
extension $\wt S$ of $S$. The following definitions can be found in
M.~Kre\u\i n and H.~Langer papers \cite{KL1}, \cite{KL2}, \cite{KL}
for a densely defined $S$ and in the Langer and Textorius paper
\cite{LT} for the general case of a symmetric linear relation $S$.
\begin{definition}
\label{GFIELD} Let $\cH$ be a Hilbert space whose dimension is equal
to the deficiency number of $S$. The function
$$\rho(\wt S)\ni z\mapsto\Gamma( z)\in\bL(\cH,\sH)$$
is called the $\gamma$-field, corresponding to $\wt S$ if
\begin{enumerate}
\item the operator $\Gamma( z)$ isomorphically maps $\cH$ onto $\sN_ z$ for all $ z\in
\rho(\wt S)$,
\item  for every $ z,\,\zeta\in\rho(\wt S)$ the  identity
\begin{equation}
\label{gg} \Gamma( z)=\Gamma(\zeta)+( z-\zeta) (\wt S- z
I)^{-1}\Gamma(\zeta)
\end{equation}
holds.
\end{enumerate}
\end{definition}
\begin{definition}
\label{QFUNCT} Let $\Gamma( z)$ be a $\gamma$-field corresponding to
$\wt S$.  An operator-valued function $Q( z)\in\bL(\cH)$ with the
property
\[
Q( z)-Q^*(\zeta)=( z-\overline \zeta)\Gamma ^*(\zeta)\Gamma( z),
\quad z,\zeta\in\rho(\wt S)
\]
is called the $Q$-function 
of $S$ corresponding to the $\gamma$-field $\Gamma( z).$
\end{definition}
The $\gamma$-field corresponding to $\wt S$ can be constructed as
follows: fix $\zeta_0\in\rho(\wt S)$ and let
$\Gamma_{\zeta_0}\in\bL(\cH,\sH)$ be a bijection from $\cH$ onto
$\sN_{\zeta_0}$. Then clearly the function
\[
\Gamma( z)=(\wt S-\zeta_0 I)(\wt S- z I )^{-1}\Gamma_{\zeta_0}=
\Gamma_{\zeta_0}+( z-\zeta_0)(\wt S- z I)^{-1}\Gamma_{\zeta_0} ,\;
 z\in\rho(\wt S)
\]
is a $\gamma$-field corresponding to $\wt S$. It follows from
Definition \ref{QFUNCT} that
\[
Q( z)=C-i\IM \zeta_0\Gamma^*_{\zeta_0}\Gamma_{\zeta_0} +( z-\bar
\zeta_0)\Gamma^*_{\zeta_0}\Gamma_ z,
\]
where $C=\RE Q(\zeta_0)\in \bL(\cH)$ is a selfadjoint operator.
Thus, the $Q$-function is defined uniquely up to a bounded
selfadjoint term in $\cH$ and it is a Herglotz-Nevanlinna function.
Moreover, for every $z$, $\IM z\ne 0$,  $-i\,\IM
 z\,\left(Q( z)-Q^*( z)\right)$ is positive definite.
Hence, $-Q^{-1}( z)$, $\IM z\ne 0$, is a Herglotz-Nevanlinna
function, too. Definition \ref{QFUNCT} combined with \eqref{gg}
gives the following representation for $Q$:
\[
Q( z)=C-i\IM \zeta_0\Gamma^*_{\zeta_0}\Gamma_{\zeta_0} +( z-\bar
\zeta_0)\Gamma^*_{\zeta_0}\left(\Gamma_{\zeta_0}+( z-\zeta_0)(\wt S-
z I)^{-1}\Gamma_{\zeta_0}\right).
\]
One of the main results of the Kre\u\i n--Langer -Textorius theory
of $Q$-functions is the following statement: \textit{If
$Q$-functions of two simple closed densely defined symmetric
operators $S_1$ and $S_2$ coincide, then the operators $S_1$ and
$S_2$ are unitarily equivalent}.

This result remains valid if condition (1) in Definition
\ref{GFIELD} is replaced with a little bit weaker one: $\Gamma(z)$
is one-to-one and has dense range in $\sN_z$ at least for one (and
then for all) $z$ \cite{HSW}.

M.~Kre\u\i n and I.~Ovcharenko in their papers \cite{KrO1} and
\cite{KrO2} defined special $Q$-functions for a densely defined
closed nonnegative operator $S$ in the Hilbert space $\sH$ with
disjoint Friedrichs and Kre\u\i n extensions $S_{\rm{F}}$ and
$S_{\rm{K}}$ \cite{Kr} ($\dom S_{\rm{F}}\cap\dom S_{\rm{K}}=\dom
S)$. Let $\cH$ be a Hilbert space with $\dim \cH$ equal to the
deficiency number of $S$. Let $a\ge 0$ and let
$$ C_a:=2a\left((S_{\rm{K}}+a I)^{-1}-(S_{\rm{F}}+a I)^{-1}\right),\;
C:=C_{1}=B_M-B_\mu,
$$
where $B_M=(I-S_{\rm{K}})(I+S_{\rm{K}})^{-1}$,
$B_\mu=(I-S_{\rm{F}})(I+S_{\rm{F}})^{-1}.$ Define the
operator-valued functions $\gamma_{\rm{F}}(\lambda)$ and
$\gamma_{\rm{K}}(\lambda)$
 $$\dC\setminus\dR_+\ni\lambda\mapsto\bL(\cH,\sH),$$
 as follows
\begin{enumerate}
\item $\cran\gamma_{\rm{F}}(\lambda)=\cran \gamma_{\rm{K}}(\lambda)=\sN_\lambda$ for each $\lambda\in
\dC\setminus\dR_+$, where $\sN_\lambda:=\ker (S^*-\lambda I),$
\item $\gamma_{\rm{F}}(\lambda)-\gamma_{\rm{F}}(z)=(\lambda-z)(S_{\rm{F}}-\lambda
I)^{-1}\gamma_{\rm{F}}(z)$,
$\gamma_{\rm{K}}(\lambda)-\gamma_{\rm{K}}(z)=(\lambda-z)(S_{\rm{K}}-\lambda
I)^{-1}\gamma_{\rm{K}}(z)$,

\item
$\ran\gamma_{\rm{F}}(-a)=\ran \gamma_{\rm{K}}(-a)=\ran C^{1/2}_a$
for each $a>0$.
\end{enumerate}

The $\bL(\cH)$-valued functions $Q_{\rm{F}}(\lambda)$ and
$Q_{\rm{K}}(\lambda)$ are defined as follows:

1)
$Q_{\rm{F}}(\lambda)-Q_{\rm{F}}^*(z)=(\lambda-z)\gamma^*_{\rm{F}}(z)\gamma_{\rm{F}}(\lambda)$,
$\lambda,z\in\dC\setminus\dR_+$,

 2) $s-\lim\limits_{x\uparrow
0}Q_{\rm{F}}(x)=0$,

3)
$Q_{\rm{K}}(\lambda)-Q_{\rm{K}}^*(z)=(\lambda-z)\gamma^*_{\rm{K}}(z)\gamma_{\rm{K}}(\lambda)$,
$\lambda,z\in\dC\setminus\dR_+$,

4)$s-\lim\limits_{x\downarrow-\infty}Q_{\rm{K}}(x)=0$.

For example, one can take $\cH=\sN:=\ker (S^*+I)$ and
\[
\begin{array}{l}
\gamma^{(0)}_{\rm{F}}(\lambda):=\left(I+(\lambda+1)(S_{\rm{F}}-\lambda
I)^{-1}\right)C^{1/2}\uphar\sN,\\
\gamma^{(0)}_{\rm{K}}(\lambda):=\left(I+(\lambda+1)(S_{\rm{K}}-\lambda
I)^{-1}\right)C^{1/2}\uphar\sN.
\end{array}
\]
Then
\[
\begin{array}{l}
Q_{\rm{F}}^{(0)}(\lambda)=-2I_{\sN}+(\lambda+1)C^{1/2}\left(I+(\lambda+1)(S_{\rm{F}}-\lambda
I)^{-1}\right)C^{1/2}\uphar\sN,\\
Q_{\rm{K}}^{(0)}(\lambda)=2I_{\sN}+(\lambda+1)\,C^{1/2}\left(I+(\lambda+1)(S_{\rm{K}}-\lambda
I)^{-1}\right)C^{1/2}\uphar\sN.
\end{array}
\]

The following statement is formulated without proof in \cite{KrO2}.
\textit{Let $Q$ be an $\bL(\cH)$-valued function holomorphic on
$\dC\setminus[0,\infty)$. Then $Q$ is the $Q_{\rm{K}}$-function
($Q_{\rm{F}}$-function) of a densely defined closed nonnegative
operator if and only if the following conditions hold true:}

\textit{1) $Q^{-1}(\lambda)\in\bL(\cH)$ for each
$\lambda\in\dC\setminus[0,\infty)$;\\
\indent 2) $\lim\limits_{x\uparrow 0}(Q(x)g,g)=\infty$ for each
$g\ne 0$ $\quad$ (2') $s-\lim\limits_{x\uparrow 0}Q(x)=0$);\\
\indent 3) $s-\lim\limits_{x\downarrow -\infty}Q(x)=0$ $\quad$
(3')$\lim\limits_{x\downarrow -\infty}(Q(x)g,g)=-\infty$ for each
$g\ne 0$);\\
\indent 4) $\lim\limits_{x\downarrow -\infty}(x Q(x)g,g)=-\infty$
for each $g\ne 0$ $\quad$ (4') $s-\lim\limits_{x\downarrow
-\infty}x^{-1}Q(x)=0$). }

In this paper it is shown that this statement holds true only for
the case $\dim\cH<\infty$. More precisely, given an arbitrary closed
\textit{not necessary densely defined} nonnegative symmetric
operator $S$ with infinite defect numbers and disjoint nonnegative
selfadjoint (operator) extensions (the case $\dom S=\{0\}$ is
possible), we construct special pairs $\{\wt S_0, \wt S_1\}$ of
disjoint ($\wt S_0\cap \wt S_1= S$) nonnegative selfadjoint
extensions different from the pair $\{S_{\rm{F}}, S_{\rm{K}}\}$ and
define the corresponding $Q$-functions $\wt Q_0$ and $\wt Q_1$ of
Kre\u\i n-Ovcharenko type, i.e., possessing properties mentioned in
the above statement. Furthermore, for the case of a densely defined
nonnegative operator $S$ we construct a new general class of
\textit{positive (generalized) boundary triplets}. This class of
boundary triplets extends the notions of (ordinary and generalized)
basic boundary triplets as well as the earlier notions of positive
boundary triplets appearing in \cite{Ar2,Ar5,AHZS,DM2,GG1,Koch}. A
key assumption used in the construction is the existence of a pair
$\{\wt S_0, \wt S_1\}$ of nonnegative selfadjoint extensions of $S$
which are disjoint, i.e. $\dom \wt S_1\cap \dom \wt S_0=\dom S$, and
whose associated closed forms satisfy the inclusion
\[
 \wt{S}_0[\cdot,\cdot]\subset\wt{S}_1[\cdot,\cdot].
\]
With some further condition of the pair $\{\wt S_0, \wt S_1\}$ this
class of boundary triplets is specialized to a class of boundary
triplets leading to realization results for the classes of
$Q$-functions of Kre\u\i n-Ovcharenko type as introduced above.

In this paper we proceed on the base of the dual situation related
to a non-densely defined Hermitian contraction $B$ and its
selfadjoint contractive extensions. Recall that so-called $Q_\mu$-
and $Q_M$-functions were introduced and studied in \cite{KrO}. These
functions are associated with the \textit{extremal extensions}
$B_\mu$ and $B_M$ of $B$ which are fundamental concepts going back
to \cite{Kr}. In \cite{AHS} the $Q$-functions of Kre\u\i
n-Ovcharenko type, formally similar to $Q_\mu$- and $Q_M$-functions,
were considered and therein analogous counterexamples to the
statements of Theorem 2.2 in \cite{KrO} were given.

In the last section of this paper boundary triplet technique plays a
central role; the basic notions and some fundamental results related
to the boundary triplets, their Weyl functions, boundary relations
and their Weyl families for the adjoint of a symmetric linear
relation can be found in \cite{DM1, DM2, DHMS, DHMS1, GG1, GGK}.

\section{Basic Preliminaries}

\subsection{Closed nonnegative forms and nonnegative selfadjoint
relations}

Let $\shh=\shh[\cdot,\cdot]$ be a nonnegative form in the Hilbert
space $\sH$ with domain $\dom \shh$. The notation $\shh[h]$ will be
used to denote $\shh[h,h]$, $h \in \dom \shh$. The form $\shh$ is
\textit{closed} if
\[
 h_n \to h, \quad \shh[h_n-h_m] \to 0, \quad h_n \in
\dom \shh, \quad h \in \sH, \quad m,n \to \infty,
\]
imply that $h \in \dom \shh$ and $\shh[h_n-h] \to 0$. The form $\shh
$ is \textit{closable} if
\[
 h_n \to 0, \quad \shh[h_n-h_m] \to 0, \quad h_n \in \dom \shh
\quad \Rightarrow \quad \shh[h_n] \to 0.\\
\]
The form $\shh $ is closable if and only if it has a closed
extension, and in this case the closure of the form is the smallest
closed extension of $\shh$. The inequality $\shh_1 \geq \shh_2$ for
semibounded forms $\shh_1$ and $\shh_2$ is defined by
\begin{equation}\label{ineq0}
 \dom \shh_1 \subset \dom \shh_2, \quad \shh_1[h] \geq
\shh_2[h], \quad h \in \dom \shh_1.
\end{equation}
In particular, $\shh_1 \subset \shh_2$ implies $\shh_1 \geq \shh_2$.
If the forms $\shh_1$ and $\shh_2$ are closable, the inequality
$\shh_1 \geq \shh_2$ is preserved by their closures.

There is a one-to-one correspondence between all closed nonnegative
forms $\shh$ and all nonnegative selfadjoint relations $H$ in $\sH$
via $\dom H \subset \dom \shh$ and
\begin{equation}\label{einz}
 \shh[h,k]=(H_s h,k), \quad h \in \dom H, \quad k \in \dom \shh;
\end{equation}
here $H_s$ stands for the nonnegative selfadjoint operator part of
$H$. In what follows the form corresponding to $H$ is shortly
denoted by $H[\cdot,\cdot]$. Recall that a selfadjoint relation $H$
admits an orthogonal decomposition $H=H_{s}\oplus (\{0\}\times \mul
H)$, where $H_{s}$ is the selfadjoint operator part $H_{s}=PH$
acting on $\cdom H=\sH\ominus \mul H$ and $P$ stands for the
orthogonal projection onto $\cdom H$. The functional calculus for a
selfadjoint relation can be defined on $\dR\cup\{\infty\}$ by
interpreting $\mul H$ as an eigenspace at $\infty$; in particular,
one defines $H^{\half}=H_{s}^{\half}\oplus (\{0\}\times \mul H)$.
The one-to-one correspondence in \eqref{einz} can also be expressed
as follows
\[
    H[h,k]=(h',k), \quad \{h,h'\}\in H, \quad k \in \dom \shh,
\]
since $(h',k)=(h',Pk)=(H_sh,k)$. The one-to-one correspondence can
be made more explicit via the second representation theorem:
\begin{equation}\label{zwei2}
 H[h,k]=(H_s^\half h, H_s^\half k),  \quad h,k \in \dom H[\cdot,\cdot]=\dom H_s^\half.
\end{equation}
The formulas \eqref{einz}, \eqref{zwei2} are analogs of Kato's
representation theorems for, in general, nondensely defined closed
semi-bounded forms in \cite[Section~VI]{Ka}; see e.g.
\cite{RoBe,AHZS,HSSW07}.

Let $H_1$ and $H_2$ be nonnegative selfadjoint relations in $\sH$,
then $H_1$ and $H_2$ are said to satisfy the inequality $H_1 \ge
H_2$ if
\begin{equation}\label{drei}
 \dom H_{1,s}^{\half}\subset \dom H_{2,s}^{\half}\text{ and }
 \|H_{1,s}^{\half} h\| \ge \|H_{2,s}^{\half} h\|, \quad
 h \in \dom H_{1,s}^{\half}.
\end{equation}
This means that the closed nonnegative forms $H_1[\cdot,\cdot]$ and
$H_2[\cdot,\cdot]$ generated by $H_1$ and $H_2$ satisfy the
inequality $H_1[\cdot,\cdot] \geq H_2[\cdot,\cdot]$; see
\eqref{ineq0}, \eqref{zwei2}.

Given a form $\shh_1$ one can generate a class of forms by means of
bounded operators.

\begin{lemma}\label{lem2.1}
Let $\shh$ be a nonnegative form with $\dom \shh\subset \sH$ and let
$C$ be a bounded operator in $\sH$. Then
\[
 \shh^C[h,k]=\shh[Ch,Ck]
\]
is also a nonnegative form. Moreover, if $\shh$ is closed or
closable the same is true for $\shh^C$.
\end{lemma}
\begin{proof}
It is clear that $\shh^C$ defines a nonnegative form in $\sH$ whose
domain is the preimage $C^{-1}(\ran C\cap \dom \shh)$, so that $\dom
\shh^C$ can be even a zero subspace. Now assume that $\shh$ is
closed and let $h_n \in \dom \shh^C$ with $h_n\to h$ and
$\shh^C[h_n-h_m]=\shh[Ch_n-Ch_m] \to 0$. Since $C$ is bounded
(continuous) $Ch_n\to Ch$ and by closability of $\shh$ one concludes
that $Ch\in \dom \shh$ and $\shh[Ch_n-Ch] \to 0$. Consequently,
$h\in \dom \shh^C$ and $\shh^C[h_n-h] \to 0$ and thus $\shh^C$ is
closed.

Similarly one proves that $\shh^C$ is closable whenever $\shh$ is
closable.
\end{proof}

The next result gives various characterizations for the inequality
$H_1 \ge H_2$; it can be viewed as an extension of Douglas
factorization in the present situation of linear relations, cf.
\cite{Doug}.

\begin{proposition}\label{douglas}
Let $H_1$ and $H_2$ be nonnegative selfadjoint relations in $\sH$.
Then the following statements are equivalent:
\begin{enumerate}\def\labelenumi{\rm (\roman{enumi})}
\item  $H_1 \ge H_2$;

\item there exists a contraction $C\in \bL(\cH)$ with $\ran C\subset \cdom H_2$
and $\;\ker H_1\subset \ker C$ such that
\[
 C H_{1}^\half \subset H_{2}^\half \quad (\Leftrightarrow\quad
 H_{2}^\half \subset H_{1}^\half C^*),
\]
in fact with these conditions $C$ is uniquely determined and it
satisfies also the following inclusions
\[
 C H_{1,s}^\half \subset H_{2,s}^\half,
 \quad \ran C^*\subset \cdom H_1\ominus \ker H_1,
 \quad \mul H_2\oplus \ker H_2 \subset \ker C^*;
\]

\item there exists a contraction $C\in [\cH]$ with $\ran C\subset \cdom H_2$
and $\ker H_1\subset \ker C$ such that
\begin{equation}\label{H12}
  (P_1H_{2,s}^\half h,P_1H_{2,s}^\half k) = H_{1,s}[C^*h,C^*k], \quad h,k \in \dom H_{2,s}^\half,
\end{equation}
where the form $[H_{1,s}^{C^*}]$ is as defined in Lemma~\ref{lem2.1}
(see also \eqref{zwei2});

\item there exists a contraction $C_1\in [\sH]$ with $\ran C_1\subset \cdom H_2$
such that
\[
 (H_1+I)^{-\half} = (H_2+I)^{-\half}C_1,
\]
here $C_1$ is uniquely determined and $\ker C_1=\mul H_1$;

\item for some nonnegative contraction $M$, $0\le M\le I$, with $\ran M\subset \cdom H_2$
one has
\[
 (H_1+I)^{-1} = (H_2+I)^{-\half}M (H_2+I)^{-\half};
\]

\item the inequality $(H_1+I)^{-1} \leq (H_2+I)^{-1}$ holds.
\end{enumerate}
\end{proposition}

\begin{proof}
Since $H_j$ is selfadjoint it admits an orthogonal decomposition
$H_j=H_{s,j}\oplus (\{0\}\times \mul H_j)$, where $H_{s,j}=P_jH_j$
is the selfadjoint operator part acting on $\cdom H_j=\sH\ominus
\mul H_j$ and $P_j$ stands for the orthogonal projection onto $\cdom
H_j$; $j=1,2$.

(i) $\Rightarrow$ (ii) Let $f\in \dom H_1^{\half}$ and define $C_0$
by setting $C_0H_{1,s}^{\half}f=H_{2,s}^{\half}f$. Then the
inequality in \eqref{drei} shows that
$\|C_0H_{1,s}^{\half}f\|=\|H_{2,s}^{\half}f\|\le
\|H_{1,s}^{\half}f\|$, and hence $C_0$ is a well-defined and
contractive operator, which can be continued to a contraction from
the closed subspace $\cran H_{1,s}$ into the closed subspace $\cran
H_{2,s}$. By extending $C_0$ to $\sH\ominus \cran H_{1,s}$ as a zero
operator gives a contractive operator $C\in[\cH]$ with $\ran
C\subset \cdom H_2\ominus \ker H_2$ and $\mul H_1\oplus \ker
H_1\subset \ker C$. The last two inclusion are equivalent to the
inclusions stated for $\ker C^*$ and $\ran C^*$ in (ii). Moreover,
by construction $CH_{1,s}^{\half}=C_0H_{1,s}^{\half}\subset
H_{2,s}^{\half}$ and
$CH_{1}^{\half}=CP_1H_{1}^{\half}=CH_{1,s}^{\half}$, so that
$CH_{1}^{\half}\subset H_{2,s}^{\half}\subset H_{2}^{\half}$. By
boundedness of $C$, $C H_{1}^\half \subset H_{2}^\half$ is
equivalent to $H_{1}^\half C^*\supset H_{2}^\half $.

Finally it is shown that the conditions $\ran C\subset \cdom H_2$,
$\ker H_1\subset \ker C$, and $C H_{1}^\half \subset H_{2}^\half$
determine $C$ uniquely. From the first and third condition one
obtains
\[
 C H_{1}^\half=P_2C H_{1}^\half \subset P_2H_{2}^\half=H_{2,s}^{\half}
\]
and this implies that $\mul H_{1}\subset \ker C$. Hence $C
H_{1}^\half=C P_1 H_{1}^\half= CH_{1,s}^{\half}\subset
H_{2,s}^{\half}$ and now the condition $\ker H_1\subset \ker C$
implies that $C$ restricted to the subspace $\cdom H_1\ominus\ker
H_1$ is uniquely determined by the condition
$CH_{1,s}^{\half}\subset H_{2,s}^{\half}$. It coincides with the
closure of $C_0$ on $\cdom H_1\ominus\ker H_1$ and is a zero
operator on the orthogonal complement $\mul H_1\oplus \ker H_1$.

(ii) $\Rightarrow$ (i) This implication is obtained directly by
applying the definition in \eqref{drei}.

(ii) $\Leftrightarrow$ (iii) If $C$ is as in (ii) then $H_{2}^\half
\subset H_{1}^\half C^*$ implies that $P_1H_{2}^\half \subset
P_1H_{1}^\half C^*$ and in view of $\cdom H_1\subset \cdom H_2$ this
leads to $P_1H_{2,s}^\half \subset H_{1,s}^\half C^*$ and
\eqref{H12}.

Conversely, if \eqref{H12} holds then $P_1H_{2,s}^\half \subset
H_{1,s}^\half C^*\subset H_{1}^\half C^*$ and taking adjoints in
$\sH$ it is easy to check that
\[
 H_{2}^\half P_1 =(H_{2,s}^\half)^* P_1 = (P_1H_{2,s}^\half)^*
 \supset (H_{1}^\half C^*)^* \supset CH_{1}^\half,
\]
which implies that $CH_{1}^\half\subset H_{2}^\half$.

(i) $\Leftrightarrow$ (vi) Recall that (i) is equivalent to
$H_1^{-1}\le H_2^{-1}$ and hence also to $H_1+I \ge H_2+I$ and
$(I+H_1)^{-1}\le (I+H_2)^{-1}$.

(ii), (vi) $\Rightarrow$ (iv) Apply (ii) to the inequality
$(I+H_1)^{-1}\le (I+H_2)^{-1}$ with $C_1=C^*$; here the second
inclusion from (ii) holds as an equality $(H_1+I)^{-\half} =
(H_2+I)^{-\half}C_1$ due to boundedness. Moreover, $\ran C_1\subset
\sH\ominus \ker((H_2+I)^{-\half})=\cdom H_2$ clearly implies that
$\ker C_1=\ker(H_1+I)^{-\half}=\mul H_1$.

(vi) $\Rightarrow$ (v) Write $(H_1+I)^{-1} =
(H_2+I)^{-\half}C_1((H_2+I)^{-\half}C_1)^*=(H_2+I)^{-\half}C_1C_1^*(H_2+I)^{-\half}$
and take $M=C_1C_1^*$.

(v) $\Rightarrow$ (vi) This is clear.
\end{proof}

Observe that if $\cdom H_1=\cdom H_2$, then \eqref{H12} can be
expressed using the forms corresponding to $H_1$ and $H_2$ in the
following simpler form:
\[
 H_2[\cdot,\cdot] \subset H_{1}^{C^*}[\cdot,\cdot].
\]
Notice also that for any fixed $t>0$ the conditions (iii) and (v)
can be also replaced by the equivalent conditions $(H_1+t)^{-\half}
= (H_2+t)^{-\half}C_t$, $\|C_t\|\le 1$, and $(H_1+I)^{-1} \leq
(H_2+I)^{-1}$, respectively; see \cite[Lemma~3.2]{HSSW07}.

To an arbitrary nonnegative l.r. $S$ in $\sH$ one can associate the
following Cayley transform
\begin{equation}\label{Cayley}
  S\mapsto B=\cC(S)=-I+2(I+S)^{-1}
  =\left\{\{f+f',f-f'\},\;\{f,f'\}\in S\right\};
\end{equation}
if $S$ is an operator then \eqref{Cayley} can be rewritten in the
form $\cC(S)=(I- S)(I+ S)^{-1}$. The Cayley transform \eqref{Cayley}
establishes a one-to-one correspondence between all nonnegative
symmetric (selfadjoint) relations $S$ and all (graphs of) Hermitian
(selfadjoint, respectively) contractions $B$ with inverse transform
\begin{equation}\label{Cayley2}
 B\mapsto S=\cC(B)=(I- B)(I+ B)^{-1}=\left\{\{(I+ B)h,
(I- B)h\}:\;h\in \sH\,\right\}.
\end{equation}
For the proof of the next statement, see \cite{AHS}.

\begin{lemma}
\label{l} Let $\wt{S}$ be a nonnegative selfadjoint relation and let
$\wt B=\cC(\wt S)$ be its Cayley transform. Then
\[
\begin{array}{l}
 \cD[\wt{S}]=\ran(I+\wt B)^{1/2};\\
 \wt{S}[u,v]=-(u,v)+2\left((I+\wt B)^{(-1/2)}u,(I+\wt B)^{(-1/2)}v\right),
\quad u,v\in \cD[\wt{S}];\\
 \cD[\wt{S}^{-1}]=\ran(I-\wt B)^{1/2};\\
 \wt{S}^{-1}[f,g]=-(f,g)+2\left((I-\wt B)^{(-1/2)}f,(I-\wt B)^{(-1/2)}g\right),
 \quad f,g\in \cD[\wt{S}^{-1}].
\end{array}
\]
\end{lemma}
If $\wt S$ is a nonnegative selfadjoint relation, then the form
domain $\cD[\wt S]$ is a Hilbert space with respect to the inner
product
\begin{equation}
\label{inpr}
 (f,g)_{\wt S}:=\wt S[f,g]+(f,g).
\end{equation}
Observe, that if $\wt B=\cC(\wt S)$ then Lemma \ref{l} shows that
\begin{equation}
\label{cinpr}
(f,g)_{\wt S} 
 =2\left((I+\wt B)^{(-1/2)}f,(I+\wt B)^{(-1/2)}g\right),\quad
f,g\in\cD[\wt S]=\ran (I+\wt B)^{1/2}.
\end{equation}

\subsection{Kre\u\i n shorted operators}

For every nonnegative bounded operator $\cS$ in the Hilbert space
$\cH$ and every subspace $\cK\subset \cH$ M.G.~Kre\u{\i}n \cite{Kr}
defined the operator $\cS_{\cK}$ by the relation
\[
 \cS_{\cK}=\max\left\{\cZ\in \bL(\cH):\,
    0\le \cZ\le \cS, \, {\ran}\cZ\subseteq{\cK}\,\right\}.
\]
The equivalent definition
\[
 \left(\cS_{\cK}f, f\right)=\inf\limits_{\f\in \cK^\perp}\left\{\left(\cS(f + \varphi),f +
 \varphi\right)\right\},
\quad  f\in\cH.
\]
Here $\cK^\perp:=\cH\ominus{\cK}$. The properties of $\cS_{\cK}$,
were studied by M.~Kre\u{\i}n and by other authors (see \cite{ARL1}
and references therein).
 $\cS_{\cK}$ is called the \textit{shorted
operator} (see \cite{AD}, \cite{AT}). It is proved in \cite{Kr} that
$S_{\cK}$ takes the form
\[
 \cS_{\cK}=\cS^{1/2}P_{\Omega}\cS^{1/2},
\]
 where $P_{\Omega}$ is the orthogonal projection in
$\cH$ onto the subspace
\[
 \Omega=\{\,f\in \cran \cS:\,\cS^{1/2}f\in {\cK}\,\}=\cran \cS\ominus
\cS^{1/2}\cK^\perp.
\]
Moreover \cite{Kr},
\begin{equation}
\label{Sh2}
 {\ran}\cS_{\cK}^{1/2}={\ran}\cS^{1/2}P_\Omega={\ran}\cS^{1/2}\cap{\cK}.
\end{equation}
It follows that
\[
  \cS_{\cK}=0 \iff  \ran \cS^{1/2}\cap \cK=\{0\}.
\]
A bounded selfadjoint operator $\cS$ in $\cH$ has the block-matrix
form
\[
\cS=\begin{pmatrix}\cS_{11}&\cS_{12}\cr \cS^*_{12}&\cS_{22}
\end{pmatrix}:\begin{array}{l}\cK\\\oplus\\\cK^\perp \end{array}\to
\begin{array}{l}\cK\\\oplus\\\cK^\perp \end{array}.
\]
It is well known (see \cite{KrO}) that
 \textit{the
operator $\cS$ is nonnegative if and only if
\[
\cS_{22}\ge 0,\; \ran \cS^*_{12}\subset\ran \cS^{1/2}_{22},\,\;
\cS_{11}\ge
\left(\cS^{-1/2}_{22}\cS^*_{12}\right)^*\left(\cS^{-1/2}_{22}\cS^*_{12}\right)
\]
and the operator $\cS_\cK$ is given by the block matrix
\begin{equation}
\label{shormat1}
\cS_\cK=\begin{pmatrix}\cS_{11}-\left(\cS^{-1/2}_{22}\cS^*_{12}\right)^*\left(\cS^{-1/2}_{22}\cS^*_{12}\right)&0\cr
0&0\end{pmatrix},
\end{equation}
} where $\cS^{-1/2}_{22}$ is the Moore-Penrose pseudo-inverse.
 If $\cS^{-1}_{22}\in\bL(\cK^\perp)$ then
\[
\cS_\cK=\begin{pmatrix}\cS_{11}-\cS_{12}\cS^{-1}_{22}\cS^*_{12}&0\cr
0&0\end{pmatrix}
\]
and the operator $\cS_{11}-\cS_{12}\cS^{-1}_{22}\cS^*_{12}$ is
called the Schur complement of the matrix $\cS$. From
\eqref{shormat1} it follows that
\[
\cS_\cK=0\iff \ran \cS^*_{12}\subset\ran
\cS^{1/2}_{22}\quad\mbox{and}\quad
\cS_{11}=\left(\cS^{-1/2}_{22}\cS^*_{12}\right)^*\left(\cS^{-1/2}_{22}\cS^*_{12}\right).
\]
\subsection{Selfadjoint contractive extensions of a nondensely defined Hermitian contraction}

Let $B$ be a Hermitian contraction in $\sH$ defined on the subspace
$\sH_0$, i.e., $(Bf,g)=(f,Bg)$ for all $f,g\in\sH_0$ and $\|B\|\le
1$. Set $\sN=\sH\ominus\sH_0$. A description of all selfadjoint
contractive ($sc$-)extensions of $B$ in $\sH$ was given by
M.G.~Kre\u{\i}n \cite{Kr}. In fact, he showed that all
$sc$-extensions of $B$ form an operator interval $[B_\mu, B_M]$,
where the extensions $B_\mu$ and $B_M$ can be characterized by
\begin{equation}
\label{extr} \left(I+B_\mu\right)_{\sN}=0, \quad
\left(I-B_M\right)_{\sN}=0,
\end{equation}
respectively. The operator $B$ admits a unique $sc$-extension if and
only if
\[
 \sup\limits_{\f\in\dom B}\cfrac{|(B\f,h)|^2}{||\f||^2-||B\f||^2}=\infty
\]
for all $h\in\sN\setminus\{0\}.$

 A description of the operator interval $[B_\mu, B_M]$ is given by
 the following equality
(cf. \cite{Kr}, \cite{KrO}):
\begin{equation}
\label{param1}
 \wt{B} = (B_M + B_\mu)/2 + (B_M - B_\mu)^{1/2} \wt Z (B_M - B_\mu)^{1/2} /2,
\end{equation}
where $\wt Z$ is a $sc$-operator in the subspace $\cran(B_M -
B_{\mu})\subseteq{\sN}$. It follows from \eqref{extr}, for instance,
that for every $sc$-extension $\wt B$ of $B$ the following
identities hold:
\[
 (I-\wt B)_{\sN}=B_M-\wt B,
\quad
 (I+\wt B)_{\sN} =\wt B-B_\mu,
\]
cf. \cite{Kr}. Hence, according to \eqref{Sh2}
\begin{equation}
\label{novrav}
\begin{array}{l}
 \ran(I-\wt B)^{1/2}\cap{\sN}=\ran(B_M-\wt B)^{1/2},\\
 \ran(I+\wt B)^{1/2}\cap{\sN}=\ran(\wt B-B_\mu)^{1/2}.
\end{array}
\end{equation}
\subsection{Nonnegative linear relations and their nonnegative selfadjoint extensions}
Let $S$ be a nonnegative l.r. in $\sH$. Recall the definition of the
Friedrichs extension $S_{\rm{F}}$ of $S$ (see \cite{Ka} for the case
of densely defined $S$ and \cite{RoBe} for nonnegative l.r. case):
$S_{\rm{F}}$ is the unique selfadjoint relation associated with the
closure of the form $S(\f,\psi)=(\f',\psi)$, $\{\f,\f'\}\in S$,
$\psi \in\dom S$:
\begin{equation}
\label{formclos}
 S_{\rm{F}}[\cdot,\cdot]=S[\cdot,\cdot]:=\clos S(\cdot,\cdot).
\end{equation}

Consider the Cayley transform $B=\cC(S)$ of $S$ in \eqref{Cayley}.
Then $B$ is a Hermitian contraction in $\sH$ and the formulas
\[
\wt B=-I + 2(I+\wt S)^{-1},\quad  \wt S=(I-\wt B)(I+\wt B)^{-1}
\]
establish a one-to-one correspondence between $sc$-extensions $\wt
B$ of $B$ and nonnegative selfadjoint extensions $\wt S$ of $S$. In
his famous paper \cite{Kr} M.G. Kre\u{\i}n proved, with $S$ being
densely defined in $\sH$, that the Cayley transform of the left
endpoint $B_\mu$ of the operator interval $[B_\mu, B_M]$ coincides
with the Friedrichs extension $S_{\rm{F}}$ of $S$, i.e.,
\[
 S_{\rm{F}}=(I-B_\mu)(I+B_\mu)^{-1}.
\]
This equality remains valid when $S$ is a l.r.; see
\cite{Ar4,CS,HMS}. Notice that $\cD[S]=\cD[S_{\rm{F}}]$. In addition
\begin{enumerate}
\item
if $S$ is a densely defined operator, then $S_F$ is characterized by
\[
\dom S_{\rm{F}}=\dom S^*\cap\cD[S];
\]
\item if $S$ is a nondensely defined operator, then
\[
S_{\rm{F}}=\left\{\,\{f, S_{0{\rm{F}}}f+h\}:\, f\in\dom S_{0F},\;
h\in\sH\ominus\sH_0\,\right\},
\]
where $S_{0F}$ stands for the Friedrichs extension of the
nonnegative operator $S_0:=P_{\sH_0}S$ having a dense domain in
$\sH_0=\cdom S$.
\end{enumerate}

Let $z\in\dC\setminus\dR_+$ and let $\sN_z=\sH\ominus\ran (S^*-\bar
z I)$ be the defect subspace of $S$ at $z$. Recall that
\[
\cD[S]\cap\sN_z=\{0\}, \quad z\in\dC\setminus\dR_+;
\]
see e.g. \cite{Kr,Ar4}. The Cayley transform
$S_{\rm{K}}:=(I-B_M)(I+B_M)^{-1}$ of the right endpoint possesses
the following property (see \cite{AN} for the operator case and
\cite{CS} for the case of l.r.):
\[
 S_{\rm{K}}=\left((S^{-1})_{\rm{F}}\right)^{-1}.
\]
It is a consequence of Proposition~\ref{douglas} and the formula
\eqref{Cayley} that if the Hermitian contractions $\wt B_1$ and $\wt
B_2$ satisfy the inequality $\wt B_1\le \wt B_2$, then equivalently
their Cayley transforms $\wt S_1=\cC(\wt B_1)$ and $\wt S_2=\cC(\wt
B_2)$ satisfy the reverse inequality $\wt S_1\ge \wt S_2$. It
follows that the linear relations $S_{\rm{F}}$ and $S_{\rm{K}}$ are
the maximal and minimal (in the sense of quadratic forms, see
\eqref{drei}) among all nonnegative selfadjoint extensions, i.e., if
$\wt S$ is a nonnegative selfadjoint extension of $S$, then
\begin{enumerate}
\item $\cD[S]\subset\cD[\wt S]\subseteq\cD[S_{\rm{K}}],$
\item $S[\f]\ge \wt S[\f]$ for all $\f\in\cD[S]$ and $\wt S[u]\ge
S_{\rm{K}}[u]$ for all $u\in\cD[\wt S]$.
\end{enumerate}
These inclusions and inequalities were originally established by
M.G. Kre\u\i n in \cite{Kr} for a densely defined $S$ and in
\cite{CS} for a l.r. $S$. The minimality property of $S_{\rm{K}}$ is
obtained by Ando and Nishio in \cite{AN} for nondensely defined
operator $S$.

The minimal nonnegative selfadjoint extension  $S_{\rm{K}}$ we will
call the \textit{Kre\u{\i}n-von-Neumann extension} of $S$. Recall
that $S$ admits a unique nonnegative selfadjoint extension, i.e.
$S_{\rm{K}}=S_{\rm{F}}$, if and only if for at least for one (and
then for all) $z\in\dC\setminus\dR_+$ the following condition is
fulfilled:
$$\sup\limits_{\f\in\dom S}\frac{\left|(\f,\f_{z})\right|^2}{(S \f,\f)}
=\infty\quad \mbox {for every}\quad\f_{z}\in\sN_{z}\setminus\{0\}.$$

The domain $\cD[S_{\rm{K}}]$ and $S_{\rm{K}}[u]$ can be
characterized as follows \cite{AN}, \cite{Ar4}:
\[
\begin{array}{l}
\cD[S_{\rm{K}}]=\left\{u\in \sH:
\;\sup\limits_{\f\in\dom S}\cfrac{|(S\f,u)|^2}{(S\f,\f)}<\infty\right\},\\
S_{\rm{K}}[u]=\sup\limits_{\f\in\dom(S)}\cfrac{|(S\f,u)|^2}{(S\f,\f)},\;
u\in \cD[S_{\rm{K}}].
\end{array}
\]
Observe that the form $S_{\rm{F}}[\cdot,\cdot]$ is the closed
restriction of the form $S_{\rm{K}}[\cdot,\cdot]$ and the form
$S^{-1}_{\rm{K}}[\cdot,\cdot]$ is the closed restriction of the form
$S^{-1}_{\rm{F}}[\cdot,\cdot]$. Besides (see \cite{Ar3})
\[
\begin{array}{l}
\inf\limits_{\f\in\cD[S_{\rm{F}}]} S_{\rm{K}}[f-\f]=0\quad\mbox{for all}\quad f\in\cD[S_{\rm{K}}],\\
\inf\limits_{\psi\in\cD[S^{-1}_{\rm{K}}]}
S^{-1}_{\rm{F}}[g-\psi]=0\quad\mbox{for all} \quad
g\in\cD[S^{-1}_{\rm{F}}].
\end{array}
\]

\section{Special pairs of nonnegative selfadjoint linear relations and corresponding pairs of selfadjoint contractions}

Let $\cA$ and $\cB$ be bounded selfadjoint operators which are
nonnegative and satisfy the inequality $\cA\le \cB$. In this case
Proposition~\ref{douglas} yields the following equivalences (see
also \cite{AHS} and the references therein):
\begin{enumerate}\def\labelenumi{\rm (\roman{enumi})}
\item $\cA\le \cB$;
\item $\cA=\cB^{1/2}\cZ\cB^{1/2}$, where $\cZ$ is a nonnegative selfadjoint contraction in
$\cran\cB$;
\item $\ran \cA^{1/2}=\cB^{1/2}\ran\cZ^{1/2}$, where $\cZ$ is as in (ii).
\end{enumerate}
Observe that if $0\le \cZ\le I$ then the block operator
\[
 P=\begin{pmatrix} \cZ & (\cZ-\cZ^2)^{1/2} \\ (\cZ-\cZ^2)^{1/2} & I-\cZ\end{pmatrix}
\]
satisfies $P^*=P=P^2$. In particular, this shows that $\cZ$ itself
is an orthogonal projection precisely when $\ran \cZ^{1/2}\cap\ran
(I-\cZ)^{1/2}=\ran(\cZ-\cZ^2)^{1/2}=\{0\}$. Since
\[
 \ran \cA^{1/2}\cap\ran(\cB-\cA)^{1/2}=\ran B^{1/2}\cZ^{1/2}\cap\ran B^{1/2}(I-\cZ)^{1/2}
 =\ran B^{1/2}(\cZ-\cZ^2)^{1/2},
\]
one concludes the following equivalence for $\cZ$ in (ii) and (iii):
\begin{equation}\label{Zproj}
 \cZ=\cZ^2 \quad \iff \quad \ran \cA^{1/2}\cap\ran(\cB-\cA)^{1/2}=\{0\}.
\end{equation}

Recall that for closed nonnegative forms $\shh_1 \subset \shh_2$
implies $\shh_1 \geq \shh_2$. The next proposition gives some
necessary and sufficient conditions for the inclusion $\shh_1
\subset \shh_2$ to hold by means of the Cayley transforms of
representing selfadjoint relations, and hence, can be seen as a
further specification of Proposition~\ref{douglas}.

\begin{proposition}\label{new0}
Let $\wt S_0$ and $\wt S_1$ be two nonnegative selfadjoint linear
relations and let
 \[
\graph\wt B_k=\cC(\wt S_k)=\left\{\{f+f',f-f'\},\;\{f,f'\}\in\wt
S_k\right\} =-I+2(I+\wt S_k)^{-1},\;k=0,1,
\]
be their Cayley transforms. Suppose that $\wt S_1\le \wt S_0$ or,
equivalently, that $\wt B_0\le \wt B_1$. Then the following
conditions are equivalent:
\begin{enumerate}
\def\labelenumi{\rm (\roman{enumi})}
\item the form  $\wt{S}_0[\cdot,\cdot]$ is a closed restriction of the form
       $\wt{S}_1[\cdot,\cdot]$;
\item the following equality holds
\[
 \cD[\wt S_1] \ominus_{\wt S_1}\ \cD[\wt S_0]=\ran(\wt B_1-\wt B_0)^{1/2};
\]
\item the following equality holds
\[
I+\wt B_0=(I+\wt B_1)^{1/2}\Pi(I+\wt B_1)^{1/2},
\]
where $\Pi$ is orthogonal projection acting in $\cran(I+\wt B_1)$;
\item the following equality holds
\[
\ran(I+\wt B_0)^{1/2}\cap\ran(\wt B_1-\wt B_0)^{1/2}=\{0\}.
\]
\end{enumerate}
\end{proposition}
\begin{proof}
The operators $\wt B_0$ and $\wt B_1$ are selfadjoint contractions
in $\sH$ and Lemma \ref{l} shows that (cf. \eqref{cinpr})
\begin{equation}\label{S12}
\wt S_k[(I+\wt B_k)^{1/2}f]+||(I+\wt
B_k)^{1/2}f||^2=2||f||^2,\;f\in\cran(I+\wt B_k),\; k=0,1.
\end{equation}
By Proposition~\ref{douglas} the inequality $I+\wt B_0\le I+\wt B_1$
is equivalent to the existence of a contraction $W:\cran(I+\wt
B_0)\to \cran(I+\wt B_1)$ ($\ker W=\{0\}$ in $\cD[\wt S_0]$) such
that
\begin{equation}\label{w}
 (I+\wt B_0)^{1/2}=(I+\wt B_1)^{1/2}W,
\end{equation}
in fact, $W$ is given by
\begin{equation}\label{opw}
 W=(I+\wt B_1)^{(-1/2)}(I+\wt B_0)^{1/2}:\cran(I+\wt B_0)\to \cran(I+\wt
 B_1).
\end{equation}
The identity \eqref{w} implies that
\begin{equation}\label{w2}
I+\wt B_0=(I+\wt B_1)^{1/2}WW^*(I+\wt B_1)^{1/2},\quad \wt B_1-\wt
B_0=(I+\wt B_1)^{1/2}(I-WW^*)(I+\wt B_1)^{1/2}.
\end{equation}
In particular,
\begin{equation}\label{w3}
 \ran (I+\wt B_0)^{1/2}=(I+\wt B_1)^{1/2}\ran W,\quad
\ran (\wt B_1-\wt B_0)^{1/2}=(I+\wt B_1)^{1/2}\ran D_{W^*},
\end{equation}
where $D_{W^*}=(I-WW^*)^{1/2}$.

(i) $\Rightarrow$ (iii) Suppose that the form
$\wt{S}_0[\cdot,\cdot]$ is a closed restriction of the form
$\wt{S}_1[\cdot,\cdot]$. Then it follows from \eqref{S12} that
$\|Wf\|^2=\|f\|^2$, i.e., $W$ is isometric and consequently
$\Pi:=WW^*$ appearing in \eqref{w2} is the orthogonal projection
onto the closed subspace $\ran W\subset\cran(I+\wt B_1)$.

(iii) $\Rightarrow$ (ii) It is clear from \eqref{w2} that $\cD[\wt
S_0]=\ran(I+\wt B_0)^{1/2}\subset \ran (I+\wt B_1)^{1/2}=\cD[\wt
S_1]$; cf. \eqref{w3}. Now suppose that $v\in\cD[\wt
S_1]\ominus_{\wt S_1}\cD[\wt S_0]$, i.e., that $\wt
S_1[u,v]+(u,v)=0$ for all $u\in\cD[\wt S_0]$; see \eqref{inpr}. By
Lemma~\ref{l} and \eqref{cinpr} this can be rewritten as
\[
\left((I+\wt B_1)^{(-1/2)}(I+\wt B_0)^{1/2}h,(I+\wt
B_1)^{(-1/2)}v\right)=0,\;h\in \sH,
\]
which in view of \eqref{opw} is equivalent to $W^*(I+\wt
B_1)^{(-1/2)}v=0$. This shows that
\begin{equation}
\label{dorth}
 \cD[\wt S_1] \ominus_{\wt S_1}\ \cD[\wt S_0]=(I+\wt B_1)^{1/2}\ker W^*.
\end{equation}
On the other hand, the identity in (iii) implies that
\begin{equation}
\label{parsum2}
 \wt B_1-\wt B_0=(I+\wt B_1)^{1/2}P(I+\wt B_1)^{1/2},
\end{equation}
where $P$ is the orthogonal projection from $\cran(I+\wt B_1)$ onto
$\cran(I+\wt B_1)\ominus\ran W= \ker W^*$, where $W^*$ acts on
$\cran(I+\wt B_1)$. Therefore,
\begin{equation}
\label{parsum3}
 \ran(\wt B_1-\wt B_0)^{1/2}=(I+\wt B_1)^{1/2}\ker W^*
 =\cD[\wt S_1] \ominus_{\wt S_1}\ \cD[\wt S_0].
\end{equation}

(ii) $\Rightarrow$ (i) Suppose that $\cD[\wt S_1] \ominus_{\wt S_1}\
\cD[\wt S_0]=\ran(\wt B_1-\wt B_0)^{1/2}.$ According to \eqref{w3}
one has $\ran (\wt B_1-\wt B_0)^{1/2}=(I+\wt B_1)^{1/2}\ran D_{W^*}$
which combined with \eqref{dorth} leads to
\begin{equation}
\label{iso} \ran D_{W^*}=\ker W^*.
\end{equation}
By the commutation relation $W^*D_{W^*}=D_W W^*$ the identity
\eqref{iso} gives $D_W W^*=0$ and this implies that the restriction
$W\uphar\cran W^* $ is isometric. However, $\cran W^*=\cran(I+\wt
B_1)^{1/2}$ and, thus, $W$ is isometric on $\cran(I+\wt B_1)^{1/2}$.
Now \eqref{w} and \eqref{S12} imply that $\wt S_1[u]=\wt S_0[u]$ for
all $u\in\cD[\wt S_0]=\ran (I+\wt B_0)^{1/2},$ i.e. the form
$\wt{S}_0[\cdot,\cdot]$ is a closed restriction of the form
       $\wt{S}_1[\cdot,\cdot]$.

Finally, the equivalence of (iii) and (iv) is obtained directly from
\eqref{Zproj}.
\end{proof}

Observe that if the equivalent conditions in Proposition~\ref{new0}
are satisfied, then it follows from \eqref{cinpr} and
\eqref{parsum2} that
\begin{equation}\label{norm}
 \left\|(\wt B_1-\wt B_0)^{1/2}g\right\|^2_{\wt S_1}
 =2||Pg||^2, \quad g\in \cran(I+\wt B_1).
\end{equation}

The next theorem will play an important role in the considerations
that follow; for this purpose we first state and prove the following
further result.

\begin{lemma}\label{newlem}
Let $\wt S_0$ and $\wt S_1$ be two nonnegative selfadjoint linear
relations such that $\wt S_1\le \wt S_0$ and let their Cayley
transforms $\wt B_0$ and $\wt B_1$ be connected by $(I+\wt
B_0)^{1/2}=(I+\wt B_1)^{1/2}W$, where $W$ is as defined in
\eqref{opw}. Then the associated forms satisfy the approximation
property
\begin{equation}\label{appr1}
 \inf\left\{\wt S_1[u-\f],\;\f\in\cD[\wt S_0]\right\}=0
\quad\mbox{for all}\quad u\in\cD[\wt S_1]
\end{equation}
if and only if
\begin{equation}\label{appr2}
 \ran(I-\wt B_1)^{1/2}\cap(I+\wt B_1)^{1/2}\ker W^*=\{0\},
\end{equation}
or, equivalently, $\ran\wt S_1^{1/2}=\cD[\wt{S}^{-1}_1]$ satisfies
\begin{equation}
\label{dorth2}
 \cD[\wt{S}^{-1}_1] \cap \left(\cD[\wt S_1] \ominus_{\wt S_1}\ \cD[\wt S_0]\right)=\{0\}.
\end{equation}
\end{lemma}
\begin{proof}
First assume that \eqref{appr1} is satisfied. By means of
Lemma~\ref{l} this condition can be rewritten as follows
\[
\inf\left\{\left\|(I-\wt B_1)^{1/2}f-(I-\wt
B_1)^{1/2}Wg\right\|^2,\;g\in\cran{(I+\wt B_0)}\right\}=0
\]
for all $f\in\cran(I+\wt B_1)^{1/2}$. Since $\ran(I-\wt
B^2_1)^{1/2}=\ran(I-\wt B_1)^{1/2}\cap\ran(I+\wt B_1)^{1/2}$ and
moreover $\cran(I-\wt B^2_1)^{1/2}=(\ker(I-\wt
B^2_1)^{1/2})^\perp=\cran(I-\wt B_1)^{1/2}\cap\cran(I+\wt
B_1)^{1/2}$ the previous condition is equivalent to
\[
\inf\left\{\left\|h-(I-\wt B_1)^{1/2}Wg\right\|^2,\;g\in\cran{(I+\wt
B_0)}\right\}=0\,
 \mbox{ for all }\, h\in\cran(I-\wt B^2_1)^{1/2}.
\]
This means that the orthogonal complement $\Omega^\perp$ in
$\cran(I-\wt B^2_1)^{1/2}$ of the linear manifold
\[
 \Omega:=\left\{(I-\wt B_1)^{1/2}Wg,\;g\in\cran{(I+\wt
B_0)}\right\}
\]
is equal to zero. However,
\[
\Omega^\perp=\left\{\f\in \cran(I-\wt B^2_1)^{1/2}:(I-\wt
B_1)^{1/2}\f\in\ker W^*\right\}
\]
and since $\ker W^*\subset\cran(I+\wt B_1)$ one concludes that the
condition $\Omega^\perp=\{0\}$ is equivalent to $\ran(I-\wt
B_1)^{1/2}\cap\ker W^*=\{0\}$. It remains to prove that this last
condition is equivalent to the condition in \eqref{appr2}. To see
this first assume that $\ran(I-\wt B_1)^{1/2}\cap\ker W^*=\{0\}$ and
let $g\in \ran(I-\wt B_1)^{1/2}\cap(I+\wt B_1)^{1/2}\ker W^*$. Then
$g\in\ran (I-\wt B^2_1)^{1/2}$ and hence $g=(I-\wt
B^2_1)^{1/2}u=(I+\wt B_1)^{1/2}w$ for some $u\in\cran (I-\wt B^2_1)$
and $w\in\ker W^*\subset\cran(I+\wt B_1)$. This implies that
\[
 (I+\wt B_1)^{1/2}[(I-\wt B_1)^{1/2}u-w]=0
\]
and since $(I-\wt B_1)^{1/2}u-w\in\cran(I+\wt B_1)$ one concludes
that $(I-\wt B_1)^{1/2}u=w$, which by the assumption $\ran(I-\wt
B_1)^{1/2}\cap\ker W^*=\{0\}$ implies that $(I-\wt B_1)^{1/2}u=w=0$.
Therefore, also $g=0$ and thus \eqref{appr2} follows. To prove the
converse assume that \eqref{appr2} is satisfied and suppose that
$g\in \ran(I-\wt B_1)^{1/2}\cap\ker W^*$. Then $g\in \cran(I+\wt
B_1)$ and clearly $(I+\wt B_1)^{1/2}g\in \ran(I-\wt
B_1)^{1/2}\cap(I+\wt B_1)^{1/2}\ker W^*$ from which one concludes
that $(I+\wt B_1)^{1/2}g=0$ and hence also $g=0$. This proves that
\eqref{appr1} and \eqref{appr2} are equivalent.

The equivalence of \eqref{appr2} and \eqref{dorth2} is obtained by
using Lemma~\ref{l}, which shows that $\ran(I-\wt
B)^{1/2}=\cD[\wt{S}^{-1}]$, and the formula $\cD[\wt S_1]
\ominus_{\wt S_1}\ \cD[\wt S_0]=(I+\wt B_1)^{1/2}\ker W^*$ in
\eqref{dorth}.
\end{proof}

\begin{theorem}\label{new} Let $\wt S_0$ and $\wt S_1$ be two nonnegative
selfadjoint relations. Then the following conditions are equivalent:
\begin{enumerate}
\def\labelenumi{\rm (\roman{enumi})}
\item the form $\wt{S}_0[\cdot,\cdot]$ is a closed restriction of the form
       $\wt{S}_1[\cdot,\cdot]$ and
\[
\inf\left\{\wt S_1[u-\f],\;\f\in\cD[\wt S_0]\right\}=0
\quad\mbox{for all}\quad u\in\cD[\wt S_1];
\]
\item the form  $\wt{S}^{-1}_1[\cdot,\cdot]$ is a closed restriction of the form
       $\wt{S}^{-1}_0[\cdot,\cdot]$ and
\[
\inf\left\{\wt S^{-1}_0[v-\psi],\;\psi\in\cD[\wt
S^{-1}_1]\right\}=0\quad\mbox{for all}\quad v\in \cD[\wt S^{-1}_0];
\]
\item the form  $\wt{S}_0[\cdot,\cdot]$ is a closed restriction of the form
       $\wt{S}_1[\cdot,\cdot]$ and the form  $\wt{S}^{-1}_1[\cdot,\cdot]$ is a closed restriction of the form
       $\wt{S}^{-1}_0[\cdot,\cdot]$;
\item the Cayley transforms
\[
\graph\wt B_k=\cC(\wt S_k)=\left\{\{f+f',f-f'\},\;\{f,f'\}\in\wt
S_k\right\},\;k=0,1
\]
 satisfy the conditions
\[
\ran(I+\wt B_0)^{1/2}\cap\ran(\wt B_1-\wt B_0)^{1/2}
 =\ran(I-\wt B_1)^{1/2}\cap\ran(\wt B_1-\wt B_0)^{1/2}=\{0\}.
\]
\end{enumerate}
\end{theorem}
\begin{proof}
By Proposition \ref{new0} the statement that the form
$\wt{S}_0[\cdot,\cdot]$ is a closed restriction of the form
$\wt{S}_1[\cdot,\cdot]$ is equivalent to the equality $\ran(I+\wt
B_0)^{1/2}\cap\ran(\wt B_1-\wt B_0)^{1/2}=\{0\}.$ Similarly by
applying inverses, cf. \eqref{Cayley2}, it can be seen that the
statement that the form $\wt{S}^{-1}_1[\cdot,\cdot]$ is a closed
restriction of the form $\wt{S}^{-1}_0[\cdot,\cdot]$ is equivalent
to the equality
$$\ran(I-\wt B_1)^{1/2}\cap\ran(\wt B_1-\wt B_0)^{1/2}=\{0\}.$$
This proves the equivalence (iii) $\Leftrightarrow$ (iv).

(i) $\Rightarrow$ (iv) By Lemma \ref{newlem} the condition
$\inf\left\{\wt S_1[u-\f],\;\f\in\cD[\wt S_0]\right\}=0$ for all
$u\in\cD[\wt S_1]$ is equivalent to \eqref{appr2}. On the other
hand, since $\wt{S}_0[\cdot,\cdot]$ is a closed restriction of the
form $\wt{S}_1[\cdot,\cdot]$ it follows from Proposition \ref{new0}
that
\begin{equation}\label{B1W}
 \ran(\wt B_1-\wt B_0)^{1/2}=(I+\wt B_1)^{1/2}\ker W^*,
\end{equation}
see \eqref{parsum3}. Combining \eqref{B1W} with \eqref{appr2} gives
$\ran(I-\wt B_1)^{1/2}\cap\ran(\wt B_1-\wt B_0)^{1/2}=\{0\}$.

(ii) $\Rightarrow$ (iv) The proof is similar to the proof of the
previous implication (apply inverses).

(iv) $\Rightarrow$ (i), (ii) Assume that (iv) holds. Again it
follows from Proposition \ref{new0} that \eqref{B1W} is satisfied.
This means that the second condition in (iv) coincides with the
condition \eqref{appr2} in Lemma~\ref{newlem} and, therefore, the
approximation property \eqref{appr1} in (i) is satisfied. Moreover,
by Proposition \ref{new0} the first property in (i) is equivalent to
the first condition in (iv). Hence (iv) implies (i) and likewise one
can derive (ii) from (iv).
\end{proof}

\begin{remark}
\label{novrema} If $\wt B_0$ and $\wt B_1$ ($\wt B_0\le \wt B_1$)
are $sc$-extensions of a Hermitian contraction, then it follows from
\eqref{novrav} that  the condition
\[
\ran(I+\wt B_0)^{1/2}\cap\ran(\wt B_1-\wt B_0)^{1/2}=\{0\}
\]
is equivalent to
\[
\ran(\wt B_0-B_\mu)^{1/2}\cap\ran(\wt B_1-\wt B_0)^{1/2}=\{0\},
\]
and, similarly,
\[
 \ran(I-\wt B_1)^{1/2}\cap\ran(\wt B_1-\wt B_0)^{1/2}=\{0\}
\]
is equivalent to
\[
 \ran(B_M-\wt B_1)^{1/2}\cap\ran(\wt B_1-\wt B_0)^{1/2}=\{0\}.
\]
\end{remark}
\begin{remark}
\label{paralslozh} If $F$ and $G$ are bounded nonnegative
selfadjoint operators, then the parallel sum $F:G$ can be
 defined \cite{FW}. The conditions
$F:G=0$ and $\ran F^{1/2}\cap\ran G^{1/2}=\{0\}$ are equivalent.
\end{remark}

The following theorem has been established in \cite{AHS}.
\begin{theorem}
\label{Tahs}
Let $S$ be a nonnegative symmetric linear relation. The pair $\{\wt
S_0,\wt S_1\}$ of  nonnegative selfadjoint linear relations
satisfies the conditions
\begin{equation}
\label{nonnegpair} \left\{\begin{array}{l}  \wt{S}_0\cap \wt{S}_1=
{S},\\  \mbox{\textit{the sesquilinear form }} \wt{S}_0[\cdot,\cdot]
  \mbox{\textit{ is a closed restriction of the form }}       \wt{S}_1[\cdot,\cdot],\\
  \mbox{\textit{the sesquilinear form }} \wt{S}^{-1}_1[\cdot,\cdot]       \mbox{\textit{ is a closed restriction of the form }}\wt{S}^{-1}_0[\cdot,\cdot]
\end{array}\right.
\end{equation}
if and only if the pair $\{\wt B_0,\wt B_1\}$ of selfadjoint
contractions satisfies conditions
\begin{equation}
\label{MAIN}
\begin{array}{l}
\wt B_0\le \wt B_1,\;\ker (\wt B_1-\wt B_0)=\dom B,\\
 \ran(\wt B_1-\wt B_0^{1/2})\cap\ran(\wt B_0-B_\mu)^{1/2}=\ran(\wt B_1-\wt B_0)^{1/2}\cap\ran(B_M-\wt B_1)^{1/2}=\{0\},
 \end{array}
\end{equation}
 where $ B=\cC( S),$ $\wt B_k=\cC(\wt S_k)$, $k=0,1$.
\end{theorem}

It is also shown in \cite{AHS} that if the defect numbers of $S$ are
finite, then the pair $\{S_{\rm{F}},S_{\rm{K}}\}$ of nonnegative
selfadjoint extensions of $S$ is the unique pair satisfying the
conditions \eqref{nonnegpair} and that if the defect numbers are
infinite, then there exist pairs $\{\wt S_0,\wt S_1\}$ different
from $\{S_{\rm{F}},S_{\rm{K}}\}$ with the properties
\eqref{nonnegpair}.


\section{Special pairs of selfadjoint extensions}

Let $B$ be a Hermitian contraction in $\sH$ with $\dom
B=\sH_0\subset \sH$ and let $\sN=\sH\ominus \sH_0$. In what follows
it is assumed that $\ker(B_M-B_\mu)=\dom B$. It is clear that the
pair $\{B_\mu, B_M\}$ determined by extreme extensions of the
operator interval $[B_\mu, B_M]$ satisfies all the conditions in
\eqref{MAIN}. According to \cite{AHS} there exists a pair $\{\wt B
_0,\wt B_1\}$, which is different from the pair $\{B_\mu, B_M\}$ and
satisfies the conditions \eqref{MAIN} if and only if $\dim
\sN=\infty$. We repeat here the construction from \cite{AHS}, since
it is essential also for the present paper.

\subsection{Construction of special pairs of nonnegative selfadjoint
contractions} Let $\sH$ be an infinite-dimensional separable Hilbert
space and let $\sK$ be an infinite-dimensional subspace of $\sH$
with an infinite-dimensional orthogonal complement $\sK^\perp$. Then
$\sK^\perp$ can be identified with $\sK$ and one can write $\sH$ as
a direct sum $\sH=\sK\oplus\sK$.

It is well known that there exist unbounded selfadjoint operators on
infinite dimensional Hilbert spaces $\sH$, whose (dense) domains
have a trivial intersection; see \cite{Neumann}, \cite{FW}, concrete
examples are given in \cite{CNZ}, \cite{Kosaki}. Consequently, there
exist bounded nonnegative operators $F$ and $G$ in $\sK$, such that
\[
 \cran F=\cran G=\sK
\quad\mbox{and}\quad
 \ran F \cap \ran G=\{0\}.
\]
Without loss of generality one can assume that $\|F\|<1$. Then $F$
is contractive and
$$\ker(I-F^2)=\{0\}.$$
Define
\[
  \cX=\begin{pmatrix} F^2 & 0 \\ 0 & I-F^2 \end{pmatrix},
\quad
  \sM=\left\{\begin{pmatrix} Gh \\ h \end{pmatrix}:\, h\in \sK \right\}.
\]
Then $\cX=\cX^*$ is a nonnegative contraction in $\sH$ with $\ker
\cX=\{0\}$ and $\sM$ is a closed linear subspace of $\sH$ such that
\[
\ran \cX^{1/2}\cap \sM=\{0\}.
\]
To see this assume that $v\in\ran \cX^{1/2}\cap\sM$. Then for some
$h,x,y\in \sK$ one has
\begin{equation}\label{01}
 v=\begin{pmatrix} Gh \\ h \end{pmatrix}
  =\begin{pmatrix} Fx \\ (I-F^2)^{1/2}y \end{pmatrix} .
\end{equation}
Since $\ran F\cap\ran G=\{0\}$, \eqref{01} implies that $Fx=Gh=0$.
Due to $\ker F=\ker G=\{0\}$ one obtains $x=0$, $h=0$. Consequently
$v=0$, and this proves the claim.

Next observe that
\[
 I-\cX=\begin{pmatrix} I-F^2 & 0 \\ 0 & F^2 \end{pmatrix},
\quad
 \sM^\perp=\left\{\begin{pmatrix} k \\ -Gk \end{pmatrix}:\, k\in \sK \right\}.
\]
Clearly, $\ker(I-\cX)=\{0\}$ and a similar argument as above shows
that
\begin{equation}\label{02}
 \ran(I-\cX)^{1/2}\cap \sM^\perp=\{0\}.
\end{equation}
Now consider
\[
\cY:=\cX+(I-\cX)^{1/2}P_\sM(I-\cX)^{1/2}.
\]
By definition $\cX\le \cY\le I$ and
\[
\cY-\cX=(I-\cX)^{1/2}P_\sM(I-\cX)^{1/2}.
\]
In particular, $\ker \cY=\{0\}$ and it follows from \eqref{Zproj}
that
\begin{equation}\label{03}
 \ran(I-\cY)^{1/2}\cap\ran (\cY-\cX)^{1/2}=\{0\}.
\end{equation}
Moreover, by factoring $\cY=\cY_0\cY_0^*$ with the row operator
$\cY_0=(\cX^{1/2}; (I-\cX)^{1/2}P_\sM)$ and using similar arguments
as in \eqref{Zproj} one concludes that
\begin{equation}\label{04}
 \ran \cX^{1/2}\cap\ran (\cY-\cX)^{1/2}=\{0\}.
\end{equation}
Notice that due to $\ker(I-\cX)=\{0\}$ the condition \eqref{02} is
equivalent to
\[
\ker (\cY-\cX)=\{0\}.
\]
 It is also worth to mention that (use e.g. \eqref{04})
\begin{equation}\label{intzeroo}
\ran \cX\cap\ran\cY=\{0\}.
\end{equation}
\subsection{Construction of special pairs of selfadjoint contractions and selfadjoint contractive extensions}
\label{spcontr1} Next introduce the selfadjoint contractions $\wt
Z_0$ and $\wt Z_1$ by
\begin{equation}\label{Zpair}
\wt Z_0:=2\cX-I,\; \wt Z_1:=2\cY-I.
\end{equation}
Then $\wt Z_0\le \wt Z_1$ and in view of \eqref{03} and \eqref{04}
one has
\[
 \ran(I-\wt Z_1)^{1/2}\cap\ran(\wt Z_1-\wt Z_0)^{1/2}=\{0\},\quad
 \ran(I+\wt Z_0)^{1/2}\cap\ran(\wt Z_1-\wt Z_0)^{1/2}=\{0\}.
\]
Additionally, by the construction one has
\[
\begin{array}{l}
 \ker(I+\wt Z_0)=\{0\},\quad \ker(I-\wt Z_0)=\{0\},\quad \ker(\wt Z_1-\wt Z_0)=\{0\},
\end{array}
\]
and hence also $\ker(I+\wt Z_1)=\{0\}$.

Now we are ready to make the construction of a pair $\{\wt B _0,\wt
B_1\}$ of contractions with the desired properties.

\begin{corollary}\label{cor4.1}
Let $B$ be a Hermitian contraction in $\sH$ with $\dom B=\sH_0$, let
$\sN=\sH\ominus\sH_0$, and assume that
\[
\dim\sN=\infty,\quad \ker (B_M-B_\mu)=\dom B=\sH_0.
\]
Then there exists a pair $\{\wt B_0, \wt B_1\}$ of $sc$-extensions
of $B$ with the properties \eqref{MAIN} which differs from the pair
$\{B_\mu, B_M\}$.
\end{corollary}
\begin{proof}
Let $\wt Z_0$ and $\wt Z_1$ be a pair of selfadjoint contractions in
$\sN$ as constructed in \eqref{Zpair} and define a pair of
$sc$-extensions of $B$ by means of \eqref{param1}:
\[
\wt{B}_k = (B_M + B_\mu)/2 + (B_M - B_\mu)^{1/2} \wt Z_k (B_M -
B_\mu)^{1/2} /2,\quad k=0,1.
\]
Then the pair $\{\wt B_0, \wt B_1\}$ satisfies all the conditions in
\eqref{MAIN} and, since clearly $\wt Z_0\neq -I_\sN$ and $\wt
Z_1\neq I_\sN$, one concludes that $\wt B_0\neq B_\mu$ and $\wt
B_1\neq B_M$.
\end{proof}

\subsection{Construction of special pairs of nonnegative selfadjoint operators and
nonnegative selfadjoint extensions} \label{spcontr2}
\begin{corollary}
\label{specpairse} There exist pairs $\{\wt S_0, \wt S_1\}$ of
unbounded nonnegative selfadjoint operators in $\sH$ such that
\begin{enumerate}
\item $\dom \wt S_0\cap\dom \wt S_1=\{0\}$,
\item $\dom \wt S^{1/2}_0\subset\dom \wt S^{1/2}_1$ and 
the form $\wt S_0[\cdot,\cdot]$ is the closed restriction of the
form $\wt S_1[\cdot,\cdot]$,
\item $\inf\limits_{\f\in\dom \wt S^{1/2}_0}\wt S_1[f-\f]=0$ for all $f\in\dom \wt S^{1/2}_1$.
\end{enumerate}
\end{corollary}
\begin{proof}
Consider the Cayley transforms of $\wt Z_0$ and $\wt Z_1$ as
constructed in \eqref{Zpair}:
\[
 \wt S_0:=(I-\wt Z_0)(I+\wt Z_0)^{-1}=(I-\cX)\cX^{-1},\quad
 \wt S_1:=(I-\wt Z_1)(I+\wt Z_1)^{-1}=(I-\cY)\cY^{-1}.
\]
Then $\wt S_0$ and $\wt S_1$, as well as the inverse $\wt
S_0^{-1}=\cX(I-\cX)^{-1}$, are nonnegative selfadjoint
\textit{operators}. Equality \eqref{intzeroo} means that
\[
\dom \wt S_0\cap\dom \wt S_1=\{0\}.
\]
From \eqref{Zpair}, \eqref{03}, \eqref{04}, Proposition~\ref{new0},
and Theorem \ref{new} we get that the form
\[
\begin{array}{l}
\wt S_0[\f,\psi]=\left((I-\cX)^{1/2}\cX^{-1/2}\f,(I-\cX)^{1/2}\cX^{-1/2}\psi\right),\\
\f,\psi\in\cD[\wt S_0]=\dom \cX^{-1/2}=\ran\cX^{1/2},
\end{array}
\]
is a closed restriction of the form
\[
\begin{array}{l}
\wt S_1[u,v]=\left((I-\cY)^{1/2}\cY^{-1/2}u,(I-\cY)^{1/2}\cY^{-1/2}v \right),\\
u,v\in\cD[\wt S_1]=\dom \cY^{-1/2}=\ran\cY^{1/2}
\end{array}
\]
and
\[
\inf\limits_{\f\in\dom \wt S^{1/2}_0}||\wt S_1^{1/2}(f-\f)||^2=0
\]
for all $f\in\dom \wt S^{1/2}_1.$
\end{proof}
Let $S$ be a nonnegative symmetric linear relation. It follows from
Theorems~\ref{new} and~\ref{Tahs} that one can construct pairs
$\{\wt S_0,\wt S_1\}$ of nonnegative selfadjoint extensions of $S$
satisfying the conditions \eqref{nonnegpair} by means of Cayley
transforms. For simplicity the next result is formulated for a
nonnegative symmetric operator $S$ along the lines in
Theorem~\ref{new}.

\begin{corollary}\label{cor4.2}
Let $S$ be a closed nonnegative symmetric, not necessary densely
defined, operator in the Hilbert space $\sH$ and assume that $S$
admits disjoint nonnegative selfadjoint (operator) extensions. Then
there exists a pair $\{\wt S_0,\wt S_1\}$ of nonnegative selfadjoint
extensions of $S$ such that
\begin{equation}
\label{spnon} \left\{\begin{array}{l}  \wt{S}_0\cap \wt{S}_1= {S},\\
\mbox{\textit{the sesquilinear form }} \wt{S}_0[\cdot,\cdot]
  \mbox{\textit{ is a closed restriction of the form }}       \wt{S}_1[\cdot,\cdot],\\
  \inf\left\{\wt S_1[u-\f],\;\f\in\cD[\wt S_0]\right\}=0 \quad\mbox{for all}\quad u\in\cD[\wt S_1]
  \end{array}\right.
\end{equation}
Moreover, if $n_\pm(S)=\infty$ then the pair $\{\wt S_0,\wt S_1\}$
differs in general from the pair $\{S_{\rm{F}},S_{\rm{K}}\}$.
\end{corollary}
\begin{proof}
The Cayley transform $B=\cC(S)=-I+2(I+S)^{-1}$ of $S$ is a
nondensely defined Hermitian contraction with $\ker (I+B)=\{0\}$.
The disjointness assumption implies that $S_{\rm{F}}\cap
S_{\rm{K}}=S$, i.e., $S_{\rm{F}}$ and $S_{\rm{K}}$ are also disjoint
nonnegative extensions of $S$. Therefore their Cayley transforms
$B_\mu=\cC(S_{\rm{F}})$ and $B_M=\cC(S_{\rm{K}})$ satisfy the
equality $\ker(B_M-B_\mu)=\dom B$. Now it is clear that the pair
$\{B_\mu, B_M\}$ satisfies all the conditions in \eqref{MAIN}.
Moreover, if $n_\pm(S)=\infty$ then $\dim\sN=\infty$ and hence by
Corollary~\ref{cor4.1} there are also other pairs $\{\wt B_0, \wt
B_1\}$ of $sc$-extensions of $B$ satisfying the properties
\eqref{MAIN}. Finally, it follows from Theorems~\ref{new}
and~\ref{Tahs} that
\[
 \wt S_k=(I-\wt B_k)(I+\wt B_k)^{-1},\quad k=0,1,
\]
are nonnegative selfadjoint extensions of $S$ satisfying the
properties in \eqref{spnon}.
\end{proof}


\section{
$Q$- functions of Hermitian contraction corresponding to the special
pairs of selfadjoint contractive extensions}

The following classes of $Q$-functions of a nondensely defined
Hermitian contraction $B$ with $\dom B=\sH_0\subset \sH$, associated
to the pair $\{\wt B_0,\wt B_1\}$ of $sc$-extensions of $B$ in $\sH$
which satisfy the conditions in \eqref{MAIN}, were introduced and
studied in \cite{AHS}:
\begin{equation}
\label{q1} {\wt Q}_0(\lambda)=\left[(\wt B_1-\wt B_0)^{1/2}(\wt
B_0-\lambda I)^{-1}(\wt B_1-\wt B_0)^{1/2}+I\right] \uphar{\sN},
\end{equation}
\begin{equation}
\label{q2} {\wt Q}_1(\lambda)=\left[(\wt B_1-\wt B_0)^{1/2}(\wt
B_1-\lambda I)^{-1}(\wt B_1-\wt B_0)^{1/2}-I\right]\uphar{\sN},\quad
\lambda\in \Ext[-1,1]
\end{equation}
These functions belongs to the Herglotz-Nevanlinna class. It is easy
to verify that
\begin{equation}
\label{inv11} {\wt Q}_0(\lambda){\wt Q}_1(\lambda)={\wt
Q}_1(\lambda){\wt Q}_0(\lambda)=-I_{\sN}, \quad \lambda\in
\Ext[-1,1].
\end{equation}
Moreover, the function $\wt Q_0$ possesses the properties
\begin{equation}
\label{PROP}
\begin{array}{l}
 \lim\limits_{\lambda\to\infty} {\wt Q}_0(\lambda)=I_\sN,\\
 s-\lim\limits_{\lambda\downarrow 1}\wt Q_0 (\lambda)=0,\;\lim\limits_{\lambda\uparrow -1}(\wt Q_0(\lambda)h,h)=+\infty,\; h\in\sN\setminus\{0\}
\end{array}
\end{equation}
while for $\wt Q_1$ one has
\begin{equation}
\label{PROP1}
\begin{array}{l}
 \lim\limits_{\lambda\to\infty} {\wt Q}_1(\lambda)=-I_\sN,\\
 s-\lim\limits_{\lambda\uparrow -1}\wt Q_1(\lambda)=0, \;  \lim\limits_{\lambda\downarrow 1}(\wt Q_1(\lambda)h,h)=-\infty,\; h\in\sN\setminus\{0\}.
\end{array}
\end{equation}
Observe that from \eqref{q1}, \eqref{q2}, and \eqref{inv11} follow
implications
\begin{equation}
\label{znak}
\begin{array}{l}
\lambda\in(-\infty,-1)\Rightarrow \wt Q_0(\lambda)> 0,\;\wt
Q_1(\lambda)< 0,\\
\lambda\in (1,+\infty)\Rightarrow \wt Q_1(\lambda)<0,\;\wt
Q_0(\lambda)>0.
\end{array}
\end{equation}
For the pair $\{B_\mu, B_M\}$ the corresponding $Q$-functions,
called the $Q_\mu$ and $Q_M$-functions, were originally defined and
investigated by Kre\u\i n and Ovcharenko in \cite{KrO}. It is stated
in \cite{KrO} that if the function $\wt Q_0$ ($\wt Q_1$) possesses
the properties in \eqref{PROP}, then there exists a nondensely
defined Hermitian contraction $B$ such that $\ker (B_M-B_\mu)=\dom
B$ and $\wt Q_0$ (respect., $\wt Q_1$) coincides with $Q_\mu$
(respect., with $ Q_M$). However, this statements appears to be true
only in the case that $\dim\sN<\infty$.

The class of Herglotz-Nevanlinna functions holomorphic in
$\dC\setminus[-1,1]$ and satisfying conditions \eqref{PROP}
(respect., \eqref{PROP1}) is denoted in \cite{AHS} by $\sS_\mu(\sN)$
(respect., by $\sS_M(\sN)$). Thus the function $\wt Q_0$ defined by
\eqref{q1} belongs to the class $\sS_\mu(\sN)$, while the function
$\wt Q_1(\lambda)=\wt Q^{-1}_0(\lambda)$ belongs to the class
$\sS_M(\sN)$. The next theorem, which contains a proper
characterization for the conditions stated by Kre\u\i n and
Ovcharenko in \cite{KrO}, has been established in \cite{AHS}.
\begin{theorem}
\label{Theor} Assume that ${\wt Q}\in\sS_\mu(\sN)$. Then there exist
a Hilbert space $\sH$ containing $\sN$ as a subspace, a Hermitian
contraction $B$ in $\sH$ defined on $\dom B=\sH\ominus\sN$, and a
pair $\{\wt B_0,\wt B_1\}$ of $sc$-extensions of $B$, satisfying
\eqref{MAIN} such that
\[
\begin{array}{l}
Q(\lambda)=\left[(\wt B_1-\wt B_0)^{1/2}(\wt B_0-\lambda I)^{-1}(\wt B_1-\wt B_0)^{1/2}+I\right] \uphar{\sN},\\
-Q^{-1}(\lambda)=\left[(\wt B_1-\wt B_0)^{1/2}(\wt B_1-\lambda
I)^{-1}(\wt B_1-\wt B_0)^{1/2}-I\right]\uphar{\sN}, \quad \lambda\in
\Ext[-1,1].
\end{array}
\]
If $\dim\sN<\infty$, then  necessary
$$
\left\{\begin{array}{l}\wt B_0=B_\mu\\
\wt B_1=B_M
\end{array}. \right.$$
\end{theorem}
It is emphasized that in the case $\dim \sN=\infty$ there exist
pairs different from $\{B_\mu, B_M\}$ satisfying \eqref{MAIN} and
their corresponding $Q$-functions given by \eqref{q1} and \eqref{q2}
also satisfy \eqref{PROP} and \eqref{PROP1}, giving a contradiction
to the above mentioned result in \cite{KrO} in the infinite
dimensional case $\dim\sN=\infty$.

Recall from \cite{KrO} that two Hermitian operators $B$ and $B'$
defined on the subspaces $\dom B$ and $\dom B'$ of the Hilbert
spaces $\sH=\dom B\oplus\sN$ and $\sH'=\dom B'\oplus\sN$,
respectively, are said to be \textit{$\sN$-unitarily equivalent}
\cite{AHS11}, \cite{ArlKlotz2010}, if there is a unitary operator
$U$ from $\sH$ onto $\sH'$, such that
\[
 U\uphar \sN=I_\sN,\quad U(\dom B)=\dom B',\quad UB=B'U.
\]
Moreover, $B$ in $\sH$ is said to be
 \textit{simple} if there is
no nontrivial subspace invariant under $B$. An equivalent condition
due to M.G.~Kre\u{\i}n and I.E.~Ovcharenko \cite[Lemma~2.1]{KrO} for
Hermitian contraction $B$ is that the subspace $\sN=\sH\ominus\dom
B$ is generating for some (equivalently for every) selfadjoint
extension $\wt B$ of $B$:
\[
  \sH=\cspan\{\,\wt B^n\sN:\, n=0,1,\dots \,\}=\cspan \{\,(\wt B-\lambda I)^{-1} \sN :\, |\lambda|>1\,\}.
\]
In \cite{KrO} it is shown that the simple part of the Hermitian
contraction $B$ is uniquely determined by its $Q_\mu$
($Q_M$)-function up to unitary equivalence. An analogous statement
holds for functions belonging to the classes $\sS_\mu(\sN)$ and
$\sS_M(\sN)$. Moreover, the following generalization of this result
for the pair $\{\wt B_0,\wt B_1\}$ of $sc$-extensions of $B$ is also
true.

\begin{proposition}
\label{pruniq} Let $B$ and $B'$ be simple Hermitian contractions in
$\sH=\dom B\oplus\sN$ and $\sH'=\dom B'\oplus\sN$, respectively, and
let $\wt Q_0(\lambda)$ and $\wt Q_0'(\lambda)$ be defined via
\eqref{q1} ($\wt Q_1(\lambda)$ and $\wt Q_1'(\lambda)$ be defined
via \eqref{q2}) with the pair $\{\wt B_0,\wt B_1\}$ and $\{\wt
B'_0,\wt B'_1\}$, respectively. If $\wt Q_0(\lambda)$ and $\wt
Q_0'(\lambda)$ are equal, then  $B$ and $B'$ and the pairs $\{\wt
B_0,\wt B_1\}$ and $\{\wt B'_0,\wt B'_1\}$ are unitarily equivalent
with the same unitary operator $U$.
\end{proposition}

We also recall another statement which concerns the compressed
resolvent
$$Q_{\wt B}(\lambda):= P_\sN(\wt B-\lambda I)^{-1}\uphar\sN$$
associated to a selfadjoint contraction $\wt B$ and which can also
be found from \cite{AHS11}.

\begin{proposition}
\label{PPP} Let $\wt B$ be a selfadjoint contraction in the Hilbert
space $\sH$, and let $\sN\subseteq\sH$. Suppose that $\wt B$ is
$\sN$-minimal, i.e. $\sH=\cspan \{\,(\wt B-\lambda I)^{-1} \sN :\,
|\lambda|>1\,\}$. Then the following conditions are equivalent:
\begin{enumerate}
\def\labelenumi{\rm (\roman{enumi})}
\item $\sN=\sH$;
\item the operator-valued function $Q^{-1}_{\wt B}(\lambda)+\lambda I$ is constant.
\end{enumerate}
\end{proposition}
Since
\[
\begin{array}{l}
\left[(\wt B_1-\wt B_0)^{1/2}(\wt B_0-\lambda I)^{-1}(\wt B_1-\wt B_0)^{1/2}+I\right] \uphar{\sN}\\
\qquad=\left[(\wt B_1-\wt B_0)^{1/2}Q_{\wt B_0}(\lambda)(\wt B_1-\wt
B_0)^{1/2}+I\right] \uphar{\sN},
\end{array}
\]
one can apply Proposition \ref{PPP} and see that it is possible that
$\dom B=\{0\}$ in Theorem \ref{Theor}. The example of such a
situation is provided by the pair of operators $\{\wt Z_0,\wt Z_1\}$
in $\sH$ constructed in Subsection \ref{spcontr1} and the
corresponding functions satisfy
\[
\begin{array}{l}
\wt Q_0(\lambda)= (\wt Z_1-\wt Z_0)^{1/2}(\wt Z_0-\lambda
I)^{-1}(\wt Z_1-\wt Z_0)^{1/2}+I_\sH\in \sS_\mu(\sH),\\
\wt Q_1(\lambda)=-\wt Q_0^{-1}(\lambda)= (\wt Z_1-\wt Z_0)^{1/2}(\wt
Z_1-\lambda I)^{-1}(\wt Z_1-\wt Z_0)^{1/2}-I_\sH\in \sS_M(\sH).
\end{array}
\]
As an addition to \cite{AHS} the following statement will now be
proved.

\begin{theorem}
\label{nden} Let $B$ be a Hermitian contraction in $\sH$ with $\dom
B=\sH_0\subset \sH$. Suppose $\ker(I+B)=\{0\}$. Let the pair $\{\wt
B_0,\wt B_1\}$ of $sc$-extensions of $B$ satisfy the equivalent
conditions in Proposition~\ref{new0}. Then the following conditions
are equivalent:
\begin{enumerate}
\def\labelenumi{\rm (\roman{enumi})}
\item
$s-\lim\limits_{\lambda\uparrow -1}(\lambda+1)\wt Q_0(\lambda)=0$;
\item
$ \lim\limits_{\lambda\uparrow -1}\cfrac{(\wt
Q_1(\lambda)f,f)}{1+\lambda}=-\infty,\; f\in\sN\setminus\{0\};$
\item $\ker(I+\wt B_0)=\{0\}.$
\end{enumerate}
\end{theorem}
\begin{proof}
Using \eqref{q1} together with the following well-known relations
for a nonnegative selfadjoint operator $G$
\[
\lim\limits_{y\uparrow 0}y(G-yI)^{-1}f=\left\{\begin{array}{l} 0,\quad\mbox{if}\; f\in\cran G\\
-f,\quad\mbox{if}\; f\in\ker G
\end{array}
\right.,
\]
and the identity $\cran(\wt B_1-\wt B_0)^{1/2}=\sN$ we get that
\[
(i)\iff\sN\subseteq\cran(I+\wt B_0).
\]
On the other hand, using the equivalence $\sN\subseteq\cran(I+\wt
B_0)\iff \dom B\supseteq\ker(I+ \wt B_0),$ the condition
$\ker(I+B)=\{0\}$, and the fact that $\wt B_0$ is a $sc$-extension
of $B$, we have
\[
(i)\iff (iii).
\]
Due to the equality
$$\wt Q_1(\lambda)=-\wt Q_0^{-1}(\lambda),\quad \lambda\in\Ext[-1,1],$$
we get with $\lambda<-1$
\[
\begin{array}{l}
||f||^2=\left(\wt Q_0(\lambda)f, -\wt Q_1(\lambda)\right)
 \le\sqrt{\left(\wt Q_0(\lambda)f,f\right)}\,\sqrt{\left(\wt Q_0(\lambda)\wt Q_1(\lambda)f,\wt Q_1(\lambda)f\right)}\\
 \hspace{9mm} =\sqrt{\left(\wt Q_0(\lambda)f,f\right)}\,\sqrt{-\left(f,\wt
Q_1(\lambda)f\right)}.
\end{array}
\]
It follows that
\[
-\left(f,\wt Q_1(\lambda)f\right)\ge \cfrac{||f||^4}{\left(\wt
Q_0(\lambda)f,f\right)},\quad \lambda<-1.
\]
Hence (i) $\Rightarrow$ (ii).

Next suppose that (ii) holds true. Since $\wt B_1-\wt B_0=(I+\wt
B_1)^{1/2}P(I+\wt B_1)^{1/2}$, where $P$ is an orthogonal projection
(see Proposition~\ref{new0}, \eqref{parsum2}), we get that
\begin{equation}
\label{DFV} (\wt B_1-\wt B_0)^{1/2}f=(I+\wt B_1)^{1/2}\cV f,\;
f\in\cran (\wt B_1-\wt B_0)^{1/2}=\sN
\end{equation}
where $\cV$ is an isometry from $\sN(=\cran(\wt B_1-\wt B_0)^{1/2})$
into $\cran(I+\wt B_1).$ With $\lambda<-1$ one obtains
\begin{multline*}
\left(\wt Q_1(\lambda)f,f\right)=\left((\wt B_1-\lambda I)^{-1}(I+\wt B_1)\cV f,\cV f\right)-||f||^2\\
= -(1+\lambda)\left((\wt B_1-\lambda I)^{-1}\cV f,\cV f\right),\;
f\in\sN.
\end{multline*}
Therefore
\[
\cfrac{(\wt Q_1(\lambda)f,f)}{1+\lambda}=-||(\wt B_1-\lambda
I)^{-1/2}\cV f||^2.
\]
One concludes that
$$ (ii)\iff\ran\cV\cap\ran(I+\wt B_1)^{1/2}=\{0\}.$$
From the definition of the isometry $\cV$ in \eqref{DFV} we have
\[
\ran\cV\cap\ran(I+\wt B_1)^{1/2}=\{0\}\iff \ran(I+\wt
B_1)\cap\ran(\wt B_1-\wt B_0)^{1/2}=\{0\}.
\]
With $g\in\ker (I+\wt B_0)$ the equality
\[
I+\wt B_1=I+\wt B_0+(\wt B_1-\wt B_0)
\]
yields the identity $(I+\wt B_1)g=(\wt B_1-\wt B_0)g$. Thus, (ii)
$\Rightarrow$ (iii). The proof is complete.
\end{proof}

\section{$Q$-functions of a nonnegative symmetric operator corresponding to the special pairs of
nonnegative selfadjoint extensions} Let $S$ be a closed nonnegative
symmetric operator, which is in general nondensely defined. It is
assumed that $S$ admits disjoint nonnegative selfadjoint operator
extensions. In the case of nondensely defined $S$ this yields, in
particular, that $S_{\rm{K}}$ is an operator (i.e. it has no
multi-valued part).

Let the linear fractional transformation $B$ of $S$ be defined by
\[
B:=(I-S)(I+S)^{-1}.
\]
Since $S_{\rm{F}}\cap S_{\rm{K}}= S$, we get $\ker (B_M-B_\mu)=\dom
B$. Consider two nonnegative selfadjoint operator extensions $\wt
S_0$ and $\wt S_1$ of $S$ given by
\[
\wt S_k=(I-\wt B_k)(I+\wt B_k)^{-1},\; k=0,1,
\]
where the pair of $sc$-extensions $\{\wt B_0,\wt B_1\}$ satisfies
the condition \eqref{MAIN}. Notice that
\[
\wt B_1 - \wt B_0=2\left((\wt S_1+I)^{-1}-(\wt S_0+I)^{-1}\right)
\]

Next introduce the so-called $\gamma$-fields by the formulas
\[
\begin{array}{l}
\dC\setminus\dR_+\ni\lambda\mapsto\gamma_0(\lambda):=\left(I+(\lambda+1)(\wt
S_0-\lambda
I)^{-1}\right)(\wt B_1 - \wt B_0)^{1/2}\uphar\sN\in\bL(\sN,\sH),\\
\dC\setminus\dR_+\ni\lambda\mapsto\gamma_1(\lambda):=\left(I+(\lambda+1)(\wt
S_1-\lambda I)^{-1}\right)(\wt B_1 - \wt
B_0)^{1/2}\uphar\sN\in\bL(\sN,\sH).
\end{array}
\]
Then define
\begin{equation}
\label{qf00}
\wt \cQ_0(\lambda)=-I_{\sN}+\frac{\lambda+1}{2}(\wt B_1 - \wt
B_0)^{1/2}\left(I+(\lambda+1)(\wt S_0-\lambda I)^{-1}\right)(\wt B_1
- \wt B_0)^{1/2}\uphar\sN,\;\lambda\in\dC\setminus\dR_+,
\end{equation}
\begin{equation}
\label{qf01} \wt \cQ_1(\lambda)=I_{\sN}+\frac{\lambda+1}{2}(\wt B_1
- \wt B_0)^{1/2}\left(I+(\lambda+1)(\wt S_1-\lambda
I)^{-1}\right)(\wt B_1 - \wt
B_0)^{1/2}\uphar\sN,\;\lambda\in\dC\setminus\dR_+
\end{equation}
If $\wt B=(I-\wt S)(I+\wt S)^{-1}$ is the linear fractional
transformation of a nonnegative selfadjoint operator, then its
resolvent can be expressed in the form
\begin{equation}
\label{preobras} (\wt B-\mu
I)^{-1}=-\frac{1}{1+\mu}\left(I+\frac{2}{1+\mu}\left(\wt
S-\frac{1-\mu}{1+\mu}\, I\right)^{-1}\right).
\end{equation}
It follows that
\begin{equation}
\label{cq01}
\begin{array}{l}
\wt \cQ_0(\lambda)=-\left(I_{\sN}+(\wt B_1-\wt B_0)^{1/2}\left(\wt
 B_0-\frac{1-\lambda}{1+\lambda}\, I_\sH\right)^{-1}(\wt B_1-\wt B_0)^{1/2}\right)\uphar\sN=-\wt Q_0\left(\frac{1-\lambda}{1+\lambda}\right),\\
 \wt\cQ_1(\lambda)=-\left(-I_{\sN}+(\wt B_1-\wt B_0)^{1/2}\left(\wt
 B_1-\frac{1-\lambda}{1+\lambda}\, I_\sH\right)^{-1}(\wt B_1-\wt B_0)^{1/2}\right)\uphar\sN=-\wt Q_1\left(\frac{1-\lambda}{1+\lambda}\right),\\
  \end{array}
\end{equation}
where the functions $\wt Q_0$ and $\wt Q_1$ are given by \eqref{q1}
and \eqref{q2} with $\lambda\in\dC\setminus\dR_+$.
From \eqref{cq01} and \eqref{znak} it follows that
\[
\lambda\in
(-\infty,0)\Rightarrow\wt\cQ_0(\lambda)<0,\;\wt\cQ_1(\lambda)>0.
\]

\begin{definition}
\label{klass0} Let $\cH$ be a separable Hilbert space. Then denote
by $\sS_{\rm{F}}(\cH)$ the class of Herglotz-Nevanlinna
$\bL(\cH)$-valued functions $\cM(\lambda)$ holomorphic on
$\dC\setminus\dR_+$ and possessing the properties
\begin{enumerate}
\item $\cM^{-1}(\lambda)\in\bL(\cH)$ for all $\lambda\in\dC\setminus\dR_+,$
\item $s-\lim\limits_{x\uparrow 0} \cM(x)=0,$
\item $\lim\limits_{x\downarrow-\infty}(\cM(x)g,g)_\cH=-\infty$ for each $g\in\cH\setminus\{0\},$
\item
 $s-\lim\limits_{x\downarrow -\infty}x^{-1}\cM(x)=0$.
\end{enumerate}
\end{definition}
\begin{definition}
\label{klass1} Let $\cH$ be a separable Hilbert space. Then denote
by $\sS_{\rm{K}}(\cH)$ the class of Herglotz-Nevanlinna
$\bL(\cH)$-valued functions $\cN(\lambda)$ holomorphic on
$\dC\setminus\dR_+$ and possessing the properties
\begin{enumerate}
\item $\cN^{-1}(\lambda)\in\bL(\cH)$ for all $\lambda\in\dC\setminus\dR_+,$
\item $\lim\limits_{x\uparrow 0}(\cN(x)g,g)_\cH=+\infty$ for each $g\in\cH\setminus\{0\},$
\item $s-\lim\limits_{x\downarrow-\infty}\cN(x)=0$,
\item
 $\lim\limits_{x\downarrow -\infty}x(\cN(x)g,g)_\cH=-\infty$ for each $g\in\cH\setminus\{0\}.$
\end{enumerate}
\end{definition}

Clearly, the class $\sS_{\rm{F}}(\cH)$ is a subset of the inverse
Stieltjes class and $\sS_{\rm{K}}(\cH)$ is subset of the
 Stieltjes class of $\bL(\cH)$-valued functions  \cite{KacK}.

\begin{theorem}
\label{adth}
The function $\wt \cQ_0$ belongs to the class $\sS_{\rm{F}}(\sN)$,
while the function $\wt\cQ_1$ belongs to the class
$\sS_{\rm{K}}(\sN)$ and
$$\wt \cQ_0(\lambda)\wt\cQ_1(\lambda)=\wt\cQ_1(\lambda)\wt\cQ_0(\lambda)=-I_\sN$$
for each $\lambda\in\dC\setminus\dR_+$.
\end{theorem}
\begin{proof} The statements follow from \eqref{inv11}, \eqref{PROP}, \eqref{PROP1},
Theorem \ref{nden}, and \eqref{cq01}.
\end{proof}
\begin{theorem}
\label{susche} Let $\cH$ be a separable Hilbert space and let
$\cM\in\sS_{\rm{F}}(\cH)$ ($\cN\in\sS_{\rm{K}}(\cH)$). Then there
exists a Hilbert space $\sH$, containing $\cH$ as a subspace, a
closed simple nonnegative possibly nondensely defined operator $S$
in $\sH$, and a pair $\{\wt S_0,\wt S_1\}$ of nonnegative
selfadjoint operator extensions of $S$, satisfying \eqref{spnon} and
such that
\[
\cM(\lambda)=Y^*\wt\cQ_0(\lambda)Y\qquad
(\cN(\lambda)=Y^*\wt\cQ_1(\lambda)Y), \;
\lambda\in\dC\setminus\dR_+,
\]
where $Y\in\bL(\cH,\cH)$ is an isomorphism and $\wt\cQ_0$
($\wt\cQ_1$) is given by \eqref{qf00} (\eqref{qf01}). If
$\dim\cH<\infty$ or $\dim\cH=\infty$ but $\IM \cM(i)$ ($\IM \cN(i)$)
is positive definite, then $S$ is densely defined and the equalities
$\wt S_0=S_{\rm{F}}$, $\wt S_1=S_{\rm{K}}$ hold true.
\end{theorem}
\begin{proof} We will prove the statement for $\cM\in\sS_{\rm{F}}(\cH)$. Since the function $\cM$ belongs to the inverse Stieltjes
class, the operator $-\cM(-1)$ is positive definite. Let
$Y=\left(-\cM(-1)\right)^{1/2}$ and define
\[
\wt\cQ_0(\lambda)=Y^{-1}\cM(\lambda) Y^{-1},\;
\lambda\in\dC\setminus\dR_+,
\]
\[
\wt Q_0(z)=-\wt\cQ_0\left(\frac{1-z}{1+z}\right),\;
z\in\dC\setminus[-1,1].
\]
Due to $\cM\in\sS_{\rm{F}}(\cH)$ the function $\wt Q_0$ belongs to
the class $\sS_\mu(\cH)$ and, moreover,
\[
s-\lim\limits_{x\uparrow -1}(x+1)\wt Q_0(x)=0.
\]
By \cite[Theorem 5.1]{AHS} there exists a Hilbert space $\sH$
containing $\cH$ as a subspace, a simple Hermitian contraction $B$
defined on $\dom B=\sH\ominus\cH$ with the property
$\ker(B_M-B_\mu)=\dom B$, and a pair $\{\wt B_0,\wt B_1\}$ of
$sc$-extensions, satisfying \eqref{MAIN} such that
\[
\wt Q_0(z)=\left[(\wt B_1-\wt B_0)^{1/2}(\wt B_0-z I)^{-1}(\wt
B_1-\wt B_0)^{1/2}+I\right] \uphar{\cH},\; z\in\dC\setminus[-1,1].
\]
From Theorem \ref{nden} it follows that $\ker(I+\wt B_0)=\{0\}$.

Now define
\[
S=(I-B)(I+B)^{-1}.
\]
Then $S$ is a closed nonnegative operator, possibly nondensely
defined, and the pair $\{\wt S_0,\wt S_1\}$ of its nonnegative
selfadjoint (operator) extensions defined by
$$\wt
S_k=(I-\wt B_k)(I+\wt B_k)^{-1},\; k=0,1,$$ satisfies conditions
\eqref{spnon}. Finally, \eqref{preobras} implies that the function
\[
-I_{\cH}+\frac{\lambda+1}{2}(\wt B_1 - \wt
B_0)^{1/2}\left(I+(\lambda+1)(\wt S_0-\lambda I)^{-1}\right)(\wt B_1
- \wt B_0)^{1/2}\uphar\cH,\;\lambda\in\dC\setminus\dR_+,
\]
coincides with $\wt \cQ_0.$ Thus, $\cM(\lambda)=Y\wt
\cQ_0(\lambda)Y$ for all $\lambda\in\dC\setminus \dR_+$.

Let $\dim\cH<\infty$. Then $B$ is Hermitian contraction with finite
equal deficiency indices. In this case the pair $\{\wt B_0, \wt
B_1\}$ necessarily coincides with the pair $\{B_\mu, B_M\}$.
Moreover, $\ker(I+B_F)=\{0\}$, so that the operator $S$ is densely
defined, and the equalities $\wt S_0=S_{\rm{F}}$ and $\wt
S_1=S_{\rm{K}}$ follow.

It is clear that $\IM \cM(i)=-Y\IM Q_0(-i)Y.$ If $\IM\cM(i)$ has a
bounded inverse, then according to \cite[Corollary 6.3]{AHS} one has
$\wt B_0=B_\mu$, $\wt B_1=B_M$, and $\ran(B_M-B_\mu)=\cH$, and since
$\ker(I+S_{\rm{F}})=\{0\}$, one concludes again that the operator
$S$ is densely defined and that $\wt S_0=S_{\rm{F}}$ and $\wt
S_1=S_{\rm{K}}$.
 \end{proof}

Thus if $\cH$ is finite dimensional and $\cM\in\sS_{\rm{F}}(\cH)$,
then there exists a closed \textit{densely defined} nonnegative
operator $S$ with finite deficiency indices such that $\cM$ is the
$Q_F$-function of $S$ and $-\cM^{-1}$ is the $Q_K$-function of the
same $S$.

If $\dim\cH=\infty$, then it is possible that $\dom S=\{0\}$.
Actually, one can take the pair $\{\wt S_0,\wt S_1\}$ in $\sH$ as
given in Corollary \ref{specpairse} and define the corresponding
function
\[
\wt \cQ_0(\lambda)=-I+\frac{\lambda+1}{2}(\wt Z_1 - \wt
Z_0)^{1/2}\left(I+(\lambda+1)(\wt S_0-\lambda I)^{-1}\right)(\wt Z_1
- \wt Z_0)^{1/2},\;\lambda\in\dC\setminus\dR_+.
\]
This function belongs to the class $\sS_{\rm{F}}(\sH)$ and $-\wt
\cQ^{-1}_0(\lambda)=\wt \cQ_1(\lambda)\in\sS_{\rm{K}}(\sH)$, where
\[
\wt \cQ_1(\lambda)=I+\frac{\lambda+1}{2}(\wt Z_1 - \wt Z_0)^{1/2}
\left(I+(\lambda+1)(\wt S_1-\lambda I)^{-1}\right)(\wt Z_1 - \wt
Z_0)^{1/2},\;\lambda\in\dC\setminus\dR_+.
\]

\section{Special boundary pairs, positive boundary triplets and their Weyl functions }

In this section pairs of nonnegative selfadjoint extensions of a
nonnegative symmetric operator and the associated $Q$-functions are
investigated further by constructing specific classes of
(generalized) boundary triplets and boundary pairs suitable for
nonnegative operators. In particular, some new realization results
for the classes of $Q$-functions introduced in the previous sections
are obtained, a most appealing one concerns the class
$\sS_{\rm{F}}(\cH)$ (see Definition~\ref{klass0}) which is
established in Theorems~\ref{MT},~\ref{T1} below.

\subsection{Ordinary, generalized and positive boundary triplets}

\begin{definition} \cite{Bruk}, \cite{Koch} \cite{GG1}, \cite{GGK}.
\label{Ordboundtr} Let $S$ be a closed densely defined symmetric
operator with equal defect numbers in $\mathfrak H$. Let $\cH$ be
some Hilbert space and let $\Gamma_0$ and $\Gamma_1$ be linear
mappings of $\dom S^*$ into $\cH$. A triplet $\{H, \Gamma_0,
\Gamma_1\}$ is called a space of boundary values (s.b.v.) or an
ordinary boundary triplet for $S^*$ if

a) for all $x,y\in\dom S^*$ the Green's identity
\begin{equation}
\label{greenid}
 (S^*x, y)-(x, S^*y)=(\Gamma_1x, \Gamma_0y)_\cH-(\Gamma_0x,
\Gamma_1y)_\cH,\; x, y\in\dom S^*,
\end{equation}
 holds;

 b) the mapping
$$\dom S^*\ni x\mapsto \Gamma x= \{\Gamma_0x, \Gamma_1x\}\in \cH\times \cH$$
is surjective.
\end{definition}
Denote $\sH_+:=\dom S^*$. When equipped with the inner product
\begin{equation}
\label{S*+}
 (u,v)_+:=(u,v)+(S^*u,S^*v),
\end{equation}
$\sH_+$ becomes a Hilbert space. It follows from Definition
\ref{Ordboundtr} that $\Gamma_0$, $\Gamma_1\in\bL(\sH_+,\cH)$, and
$\ker \Gamma_k\supset\dom S,$ $k=1,2$, and, moreover, that the
operators
\[
\wt S_0=S^*\uphar\ker \Gamma_0,\;\wt S_1=S^*\uphar\ker \Gamma_1
\]
are selfadjoint extensions of $S$ which in addition are transversal:
\[
\dom S^*= \dom \wt S_0 + \dom \wt S_1.
\]

The function $M(\lambda)$ defined by
$$M(\lambda)(\Gamma_0x_\lambda)=\Gamma_1x_\lambda, \quad x_\lambda\in\mathfrak N_\lambda,$$
where $\mathfrak N_\lambda$ stands for the defect subspace of $S$ at
$\lambda$, is called the Weyl function of the boundary triplet
\cite{DM1}. With the corresponding $\gamma$-field given by
\[
\gamma(\lambda):=\left(\Gamma_0\uphar\sN_\lambda\right)^{-1}
\]
the definition of the Weyl function can be rewritten in the form
$M(\lambda)=\Gamma_1\gamma(\lambda)$.

If the operators $\Gamma_0$ and $\Gamma_1$ are defined only on a
linear manifold $\cL$ which is dense in $\cH_+$, are closable w.r.t.
norms $||\cdot||_+$ and $||\cdot||_\cH$, the Green's identity
\eqref{greenid} is valid for $x,y\in\cL$, the mapping
$\Gamma_0:\cL\to\cH$ is surjective, and the operator $\wt
S_0:=S^*\uphar\ker\Gamma_0$ is selfadjoint, then $\{\cH, \Gamma_0,
\Gamma_1\}$ is said to be a \textit{generalized boundary triplet};
see \cite{DM2}.

\begin{definition}
\label{poskoch} Let $S$ be a densely defined closed positive
definite symmetric operator in $\sH$ and let $\wt S_0$ be a positive
definite selfadjoint extension of $S$.  An ordinary boundary triplet
$\{\cH,\Gamma_0,\Gamma_1\}$ for $S^*$ is called a positive boundary
triplet corresponding to the decomposition
\[
\dom S^*=\dom \wt S_0\dot+\ker S^*
\]
if
\[
(S^*f,g)=(\wt S_0\cP_0 f,g)+(\Gamma_1 f,\Gamma_0g)_\cH,\quad
f,g\in\dom S^*,
\]
where $\cP_0$ is the projector from $\sH_+=\dom S^*$ onto $\dom\wt
S_0=\ker \Gamma_0$ parallel to $\ker S^*$.
\end{definition}

By definition $\ker \Gamma_0= \dom \wt S_0$ and, moreover,
\[
 \ker \Gamma_1=\dom S\dot+\ker S^*=\dom S_{\rm{K}}.
\]
Definition \ref{poskoch} has been proposed by A.N.~Kochube\u{\i}
\cite{Koch} (see also \cite{GG1}). To cover the general case of a
nonnegative symmetric operator $S$ the following definition was
suggested in \cite{Ar2}:
\begin{definition}
\label{posbountri} Let $S$ be a densely defined closed nonnegative
symmetric operator in $\sH$. An ordinary boundary triplet
$\{\cH,\Gamma_0,\Gamma_1\}$ for $S^*$ is called positive if the
quadratic form
\[
\omega(f,f):=(S^*f,f)-(\Gamma_1f,\Gamma_0 f)_\cH,\; f\in\dom S^*
\]
is nonnegative.
\end{definition}
It follows from Definition \ref{posbountri} that if
$\{\cH,\Gamma_0,\Gamma_1\}$ is a positive boundary triplet, then
$\wt S_0$ and $\wt S_1$ are two mutually transversal nonnegative
selfadjoint extensions of $S$ such that $\wt S_1\le \wt S_0.$
Moreover, it is proved in \cite{Ar2} that positive boundary triplets
exist if and only if the Friedrichs and Kre\u\i n extensions are
transversal. An ordinary boundary triplet for a densely defined
closed nonnegative operator $S$, which satisfies the equalities
\[
\ker \Gamma_0=S_{\rm{F}} \quad \text{and} \quad \ker
\Gamma_1=S_{\rm{K}},
\]
is called \textit{basic}; see \cite{Ar2}, \cite{AHZS}. The following
theorem has been established in \cite{Ar2}.

\begin{theorem}
\label{descallpos} Let $\{\cH,\Gamma^{(0)}_0, \Gamma^{(0)}_1\}$ be a
basic boundary triplet. Then an ordinary boundary triplet
$\{\cH',\Gamma'_0,\Gamma'_1\}$ is positive if and only if the
following equalities hold
\[
\begin{array}{l}
\Gamma'_0=W\left((I_\cH+BC)\Gamma^{(0)}_0-B\Gamma^{(0)}_1\right),\\
\Gamma'_1=W^{*-1}\left(-C\Gamma^{(0)}_0+\Gamma^{(0)}_1\right)
\end{array}
\]
for some bounded nonnegative selfadjoint operators $B$ and $C$ in
$\cH$ and a linear homeomorphism $W\in\bL(\cH,\cH')$.
\end{theorem}

Notice that in \cite{DM2} and \cite{Ar5} generalized basic boundary
triplets are constructed. In the next section a more general class
of generalized positive boundary triplets is constructed.

\subsection{Special boundary pairs and corresponding positive boundary triplets}
\subsubsection{The linear manifold $\cL$}

In the rest of this section we assume that
\[
\begin{array}{l}
 \textbf{(a)}\;S\;\mbox{is a densely defined nonnegative symmetric operator operator
 in}\;\sH,\\
 \textbf{(b)}\; \wt S_0\;\mbox{and}\; \wt S_1\;\mbox{are two nonnegative selfadjoint
 extensions of}\; S,\;\mbox{such that}\\
 \dom \wt S_1\cap \dom \wt S_0=\dom S,\\
 \textbf{(c)}\;\mbox{the form}\;  \wt{S}_0[\cdot,\cdot]\;\mbox{is a closed
restriction of the form}\;\wt{S}_1[\cdot,\cdot]
\end{array}
\]
Define the linear manifold $\cL$ by the equality
\begin{equation}
\label{L1}
\cL:=\dom\wt S_0\dot+\left(\cD[\wt S_1]\ominus_{\wt S_1}\cD[\wt S_0]\right). 
\end{equation}
Let $\sN_z$ be the defect subspace of $S$ at $z$ and denote
\begin{equation}
\label{Ldefect} \wt \sN_z:=\sN_z\cap \cL, \;z\in\Ext[0,\infty).
\end{equation}
Since  $\cD[\wt S_1]\ominus_{\wt S_1}\cD[\wt S_0]\subset \sN_{-1}$,
see Proposition~\ref{new0}, it is  clear that
\begin{equation}
\label{def-1} \cD[\wt S_1]\ominus_{\wt S_1}\cD[\wt S_0]=\wt
\sN_{-1}.
\end{equation}
Consequently,
\[
\wt \sN_z=(\wt S_0+I)(\wt S_0 -zI)^{-1}\wt \sN_{-1}=
\left(I+(z+1)(\wt S_0 -zI)^{-1}\right)\wt \sN_{-1},
\]
and one has the decompositions
\begin{equation}
\label{decomp11}
 \cL=\dom\wt S_0\dot+\wt \sN_z,\;\cD[\wt S_1]=\cD[\wt
S_0]\dot+\wt \sN_z,\;z\in\Ext[0,\infty).
\end{equation}
In particular, with $z,\xi\in\Ext[0,\infty)$ the subspaces in
\eqref{Ldefect} are connected by
\[
 \wt \sN_z=(\wt S_0-\xi I)(\wt S_0 -zI)^{-1}\wt\sN_{\xi}
 =\left(I+(z-\xi)(\wt S_0 -zI)^{-1}\right)\wt\sN_\xi.
\]


\begin{lemma}
\label{Llemma} Let $S$ and $\{\wt S_0,\wt S_1\}$ satisfy conditions
\textbf{(a), (b), (c)}. Then $\cL$ defined in \eqref{L1} satisfies
\begin{equation}
\label{L2}
\cL=\dom\wt S_1 \dot+\left(\cD[\wt S_1]\ominus_{\wt S_1}\cD[\wt S_0]\right) 
\end{equation}
and
\[
\dom\wt S_0+\dom\wt S_1\subset\cL\subset\cD[\wt S_1]\cap\dom S^*.
\]
\end{lemma}
\begin{proof}
According to Proposition~\ref{new0} one has $\cD[\wt
S_1]\ominus_{\wt S_1}\cD[\wt S_0]=\ran(\wt B_1-\wt B_0)^{1/2}$ and
since $\dom \wt S_k=\ran (I+\wt B_k)$, $k=0,1$, we have
\[
\begin{array}{l}
(I+\wt B_1)f=(I+\wt B_0)f+(\wt B_1-\wt B_0)f\in\dom \wt S_0+\ran(\wt
B_1-\wt B_0)^{1/2},\\
(I+\wt B_0)f=(I+\wt B_1)f-(\wt B_1-\wt B_0)f\in\dom \wt S_1+\ran(\wt
B_1-\wt B_0)^{1/2}.
\end{array}
\]
These identities combined with \eqref{L1} lead to the sum
representation in \eqref{L2} and since $S$ is densely defined and
$\wt S_1$ is nonnegative, the sum in \eqref{L2} is direct.

The last two inclusions in the lemma are clear from \eqref{def-1}
and \eqref{L2}.
\end{proof}

If $\wt S_0=S_{\rm{F}}$, $\wt S_1=S_{\rm{K}}$ then
$\cL=\cD[S_{\rm{K}}]\cap\dom S^*$. Moreover, in the case of
transversality one has automatically $\cL=\dom S^*=\dom \wt S_0+\dom
\wt S_1$.

\begin{proposition}
\label{zamkn} Under the assumptions in Lemma~\ref{Llemma} the
sesquilinear form
\[
\dom\wt \eta=\cL, \;\wt\eta[u,v]:=\wt S_1[u,v],\; u,v\in\cL
\]
is closed in the Hilbert space $\sH_+$.
\end{proposition}
\begin{proof}
Let $\{u_n\}$ be a sequence from $\cL$ such that
\begin{enumerate}
\item $\lim\limits_{n\to\infty}u_n=u$ in $\sH_+$,
\item $\lim\limits_{m,n\to\infty}\wt S_1[u_n-u_m]=0$.
\end{enumerate}
Due to \eqref{L1} one can write $u_n=f_n+(\wt B_1-\wt
B_0)^{1/2}g_n,\;n\in\dN$, where $f_n\in\dom\wt S_0$ and
$g_n\in\sN=\cran(\wt B_1-\wt B_0)^{1/2}$, which in view of
\eqref{norm} leads to
\[
\begin{array}{l}
\wt S_1[u_n-u_m]+||u_n-u_n||^2 \\
=\wt S_1[f_n-f_m]+||f_n-f_m||^2+||(\wt B_1-\wt B_0)^{1/2}(g_n-g_m)||^2_{\wt S_1}\\
=\wt S_1[f_n-f_m]+||f_n-f_m||^2+2||g_n-g_m||^2.
\end{array}
\]
Hence the sequences $\{f_n\}$ and $\{g_n\}$ converge in $\sH$. Let
$g:=\lim\limits_{n\to\infty}g_n$. Then $g\in\sN$ and
\[
\lim\limits_{n\to\infty}(\wt B_1-\wt B_0)^{1/2}g_n=(\wt B_1-\wt
B_0)^{1/2}g.
\]
It follows from \eqref{def-1} that
\[
S^*u_n=\wt S_0f_n-(\wt B_1-\wt B_0)^{1/2}g_n,\; n\in\dN,
\]
and hence the sequence $\{u_n\}$ converges in $\sH_+$. Consequently,
$\{f_n\}$ converges in $\sH_+$. Put
\[
f:=\lim\limits_{n\to\infty}f_n \quad\mbox{in the Hilbert space}\;
\sH_+.
\]
Then $f\in\dom\wt S_0$ and
\[
u=f+(\wt B_1-\wt B_0)^{1/2}g.
\]
Thus the vector $u$ belongs to $\cL$. Since the form $\wt
S_0[\cdot,\cdot]$ is the closed restriction of the form $\wt
S_1[\cdot,\cdot]$ we get that
\[
\lim\limits_{n\to\infty}\wt S_1[f-f_n]=\lim\limits_{n\to\infty}\wt
S_0[f-f_n]=0.
\]
Therefore,
\[
\wt S_1[u-u_n]+||u-u_n||^2=\wt S_0[f-f_n]+||f-f_n||^2+||(\wt B_1-\wt
B_0)^{1/2}(g-g_n)||^2\to 0,\; n\to\infty,
\]
and this completes the proof.
\end{proof}

It follows from Proposition \ref{zamkn} that the linear manifold
$\cL$ is a Hilbert space with respect to the inner product (cf.
\eqref{S*+})
\begin{equation}
\label{newinner} (u,v)_{\wt\eta}:=\wt S_1[u,v]+(u,v)_+.
\end{equation}
\begin{lemma}\label{l1}
The identity
\[
 \wt S_1[f,\f]=(S^*f,\f)
\]
is satisfied for all $f\in\cL$ and all $\f\in\cD[\wt S_0]$.
\end{lemma}
\begin{proof} Let $f=\psi+g$, $\psi\in\dom\wt S_0,\,g\in \cD[\wt S_1]\ominus_{\wt S_1}\cD[\wt S_0]$.
According to \eqref{def-1} $\cD[\wt S_1]\ominus_{\wt S_1}\cD[\wt
S_0]\subset \sN_{-1}$, so that $S^*g=-g$ and, therefore, $S^*f=\wt
S_0\psi-g.$ On the other hand,
\[
\wt S_1[f,\f]=\wt S_1[\psi,\f]+\wt S_1[g,\f]=\wt
S_0[\psi,\f]-(g,\f)=(\wt S_0\psi-g,\f),
\]
where the second identity follows from \eqref{inpr}. This completes
the proof.
\end{proof}

\subsubsection{Boundary pairs and $\gamma$-fields}
\begin{definition}
\label{BP} The pair $\{\cH,\Gamma_0\}$ is called a boundary pair for
$\{\wt S_0,\wt S_1\}$ if $\cH$ is a Hilbert space, $\Gamma_0$ is a
continuous linear operator from the Hilbert space $\cD[\wt S_1]$
into $\cH$, and
\[
\ker \Gamma_0=\cD[\wt S_0],\;\ran \Gamma_0=\cH.
\]
\end{definition}
Due to \eqref{decomp11} and the equality $\ker\Gamma_0=\cD[\wt S_0]$
the mapping $\Gamma_0:\wt\sN_z\to \cH$ is a bijection, the inverse
operator
\begin{equation}
\label{gz} \Gamma_0(z):=\left(\Gamma_0\uphar\wt\sN_z\right)^{-1}
\end{equation}
belongs to $\bL(\cH,\cD[\wt S_1])\cap\bL(\cH,\sH)$. Since
$||\f_z||^2_+=(1+|z|^2)||\f_z||^2$ for all $\f_z\in\sN_z$, the
operator $\Gamma_0(z)$ is continuous from $\cH$ into $\cL$ with
respect to the inner product \eqref{newinner}.
\begin{definition}\label{gfield}
Let $\{\cH,\Gamma_0\}$ be a boundary pair for $\{\wt S_0,\wt S_1\}$.
The operator valued function $\Gamma_0(z)$ defined by \eqref{gz} is
called the $\Gamma_0$-field.
\end{definition}
Since $\ker \Gamma_0=\cD[\wt S_0]$ and $\ran\Gamma_0=\cH$, one
obtains the following equality:
\begin{equation}
\label{gzxi} \Gamma_0(z)=\Gamma_0(\xi)+(z-\xi)(\wt
S_0-zI)^{-1}\Gamma_0(\xi),\; z,\xi\in\Ext [0,\infty).
\end{equation}
Therefore, the $\Gamma_0$-field is a holomorphic  function in $\Ext
[0,\infty)$ and $\ran\Gamma_0(z)=\wt\sN_z.$ In addition,
\[
s-\lim\limits_{x\downarrow-\infty}\Gamma_0(x)=0.
\]
Observe that the operator $\Gamma_0\uphar\cL$ is closed in $\sH_+$.
To see this let $\{u_n\}\subset \cL$ be a sequence such that
 \[
 u_n\to u\;\mbox{in}\; \cH_+,\; \Gamma_0 u_n\to e\;\mbox{in}\;\cH\;\mbox{when}\; n\to\infty.
 \]
Due to \eqref{decomp11} and \eqref{gz}
\[
 u_n=f_n+\Gamma_0(-1)e_n,\; \{f_n\}\subset\dom\wt S_0,\;\{e_n\}\subset\cH.
\]
Since $e_n=\Gamma_0 u_n$, $n\in\dN$, the sequence $\{e_n\}$
converges in $\cH$ to the vector $e$. Therefore the sequence
$\{\Gamma_0(-1)e_n\}$ converges to $\Gamma_0(-1)e\in\wt \sN_{-1}$ in
the Hilbert space $\cD[\wt S_1]$. Hence
$\lim\limits_{n\to\infty}\Gamma_0(-1)e_n=\Gamma_0(-1)e$ in $\sH_+$.
It follows that the sequence $\{f_n\}$ converges in $\sH_+$ to some
vector $f\in \dom\wt S_0$ and, thus, $u=f+\Gamma_0(-1)e \in\cL$,
$e=\Gamma_0u$, i.e., $\Gamma_0\uphar\cL$ is closed in $\sH_+$.

Define the $\bL(\cH)$-valued function $W(z,\xi)$ by
\begin{equation}
\label{WW} \left(W(z,\xi)h,e\right)_\cH:=\wt
S_1[\Gamma_0(z)h,\Gamma_0(\xi)e],\quad h,e\in\cH.
\end{equation}
Clearly, $W(z,\xi)$ is holomorphic in $z$, anti-holomorphic in
$\xi$, and, in addition, it is a positive definite kernel.

Let $\Gamma_0^*(z)\in\bL(H,\cH)$ be the adjoint of the operator
$\Gamma_0(z)\in \bL(\cH,H)$.
\begin{lemma}
\label{diff} The function
\[
z\Gamma_0^*(\xi)\Gamma_0(z)-W(z,\xi)
\]
does not depend on $\xi$.
\end{lemma}
\begin{proof} By definition one has
\[
z\left(\Gamma_0(z)h,\Gamma_0(\xi)e\right)-\left(W(z,\xi)h,e\right)_\cH=
\left(S^*\Gamma_0(z)h,\Gamma_0(\xi)e\right)-\wt
S_1[\Gamma_0(z)h,\Gamma_0(\xi)e].
\]
Now by adding and subtracting the term $\Gamma_0(-1)e$ in the right
side of the previous formula and taking into account that
$\Gamma_0(\xi)e-\Gamma_0(-1)e\in\dom\wt S_0$, the assertion follows
from Lemma \ref{l1}.
\end{proof}

\subsubsection{Boundary triplets and Weyl functions}
\begin{theorem}
\label{G} Let $\{\cH,\Gamma_0\}$ be a boundary pair for $\{\wt
S_0,\wt S_1\}$. Then there exists a unique linear operator
$\Gamma_1:\cL\to\cH$ such that
\begin{equation}
\label{GG}
 \wt S_1[u,v]=(S^*u,v)-(\Gamma_1u,\Gamma_0
v)_{\cH}\quad\mbox{for all}\quad u\in\cL \quad\mbox{and all} \quad
v\in\cD[\wt S_1].
\end{equation}
The operator $\Gamma_1$ is bounded from the Hilbert space $\cL$,
equipped with the inner product \eqref{newinner}, to the Hilbert
space $\cH$. Moreover,
\[
 \ker \Gamma_1=\dom \wt S_1, \quad \ran
\Gamma_1=\cH.
\]
\end{theorem}
\begin{proof} Decompose $v=\f+g$, where $\f\in\cD[\wt S_0],$ $g\in \cD[\wt S_1]\ominus\cD[\wt S_0]$.
Then Lemma \ref{l1} implies that
\[
\wt S_1[u,v]-(S^*u,v)=\wt S_1[u,\f]+\wt
S_1[u,g]-(S^*u,\f)-(S^*u,g)=\wt S_1[u,g]-(S^*u,g).
\]
By Lemma \ref{Llemma} the vector $u\in\cL$ can be represented in the
form $u=h+\psi$, where $h\in\dom\wt S_1$ and $\psi\in \cD[\wt
S_1]\ominus\cD[\wt S_0]$. This yields the equality
\begin{equation}
\label{S1-S*}
 \wt S_1[u,v]-(S^*u,v)=\wt S_1[h+\psi,g]-(\wt S_1h-\psi,g)
 =\wt S_1[\psi,g]+(\psi,g)=(\psi,g)_{\wt S_1}.
\end{equation}
Therefore, for all $v\in\cD[\wt S_1]$ one has
\[
\left|\wt S_1[u,v]-(S^*u,v)\right|=\left|(\psi,g)_{\wt
S_1}\right|\le ||\psi||_{\wt S_1}||g||_{\wt S_1} \le
C\,||\psi||_{\wt S_1}||\Gamma_0 v||_{\cH},
\]
i.e. $\wt S_1[u,v]-(S^*u,v)$ is a continuous linear functional
w.r.t. $\Gamma_0 v$ on $\cH$. It follows that there exists a linear
operator $\Gamma_1:\cL\to\cH$ such that $\wt
S_1[u,v]-(S^*u,v)=-(\Gamma_1u,\Gamma_0 v)_{\cH}$ for all $u\in\cL$
and all $ v\in\cD[\wt S_1]$.

Now with $u,v\in\cL$ one obtains (see \eqref{inpr},
\eqref{newinner})
\[
 \left|(\Gamma_1u,\Gamma_0 v)_\cH\right|=\left|(S^*u,v)-\wt S_1[u,v]\right| 
 \le \sqrt{\wt S_1[u]\,\wt S_1[v]}+||S^*u||||v||
 \le 2||u||_{\wt\eta}||v||_{\wt S_1}.
\]
This implies that
\[
||\Gamma_1u||_\cH\le \wt C||u||_{\wt\eta},\quad u\in\cL,
\]
i.e., $\Gamma_1:\cL\to \cH$ is bounded.

The equality $\ker \Gamma_1=\dom \wt S_1$ follows directly from
\eqref{GG}. In view of \eqref{S1-S*} one has
\begin{equation}
\label{Gam10}
 -(\Gamma_1\psi,\Gamma_0 g)_{\cH}=(\psi,g)_{\wt S_1}
 \quad\mbox{for all}\quad
 \psi,g\in \cD[\wt S_1]\ominus_{\wt S_1}\cD[\wt S_0].
\end{equation}
Since $\ker \Gamma_0=\cD[\wt S_0]$ and $\Gamma_0\left(\cD[\wt
S_1]\ominus_{\wt S_1}\cD[\wt S_0]\right)=\cH$, it follows that
$\cran \Gamma_1=\cH$. To see that $\Gamma_1$ is surjective assume
the converse. Then by Lemma~\ref{Llemma} there exists a normalized
sequence $\{g_n\}\subset\cD[\wt S_1]\ominus_{\wt S_1}\cD[\wt S_0]$
with $\|g_n\|_\eta=1$ such that $\Gamma_1g_n \to 0$, as $n\to
\infty$. Now boundedness of $\Gamma_0$ implies that
\[
  -(\Gamma_1 g_n,\Gamma_0 g_n)=\|g_n\|_{\wt S_1}^2  \to  0.
\]
However, here $g_n\in\sN_{-1}$ and hence the norms $\|g_n\|_{\wt
S_1}$ and $\|g_n\|_\eta$ are equivalent (see \eqref{newinner}), so
that $\|g_n\|_\eta \to 0$; a contradiction. Therefore, $\ran
\Gamma_1=\cH$.
\end{proof}

\begin{definition}
\label{BT} Let $\{\cH,\Gamma_0\}$ be a boundary pair for $\{\wt
S_0,\wt S_1\}$ and let $\Gamma_1:\cL\to\cH$ be as in \eqref{GG}.
Then $\{\cH,\Gamma_0, \Gamma_1\}$ is called a boundary triplet for
the pair $\{\wt S_0,\wt S_1\}$.
\end{definition}

Observe that the Green's identity
\[
(S^*u,v)-(u,S^*v)=(\Gamma_1 u,\Gamma_0)_\cH-(\Gamma_0
u,\Gamma_1v)_\cH,\quad u,v\in\cL,
\]
is satisfied. Due to \eqref{GG} the boundary triplet introduced in
Definition \ref{BT} is a generalization of the notion of an ordinary
positive boundary triplet (see Definitions \ref{poskoch} and
\ref{posbountri}). Moreover, since \ran $\Gamma_0=\cH$ and $\wt
S_0:=S^*\uphar\ker\Gamma_0$ is a selfadjoint extension of $S$, this
is a generalized boundary triplet for $S^*$ in the sense of
\cite{DM2}.

The main result in this section connects the boundary triplet in
Definition~\ref{BT} to the study of boundary relations in
\cite{DHMS1}.

\begin{theorem}
\label{MT} Let $\{\cH,\Gamma_0, \Gamma_1\}$ be a boundary triplet
for the pair $\{\wt S_0,\wt S_1\}$ as in Definition~\ref{BT}. Then
the operator $\wt \cA$ defined by
\begin{equation}
\label{A} \wt \cA\begin{pmatrix}u\cr \Gamma_0 u
\end{pmatrix}=\begin{pmatrix}S^* u\cr -\Gamma_1 u \end{pmatrix},\; u\in\cL.
\end{equation}
is a nonnegative selfadjoint extension of $S$ acting in the Hilbert
space $\wt\sH=\sH\oplus\cH$. Moreover,
\begin{equation}
\label{DomForm}
\begin{array}{l}
\cD[\wt \cA]=\left\{\begin{pmatrix}v\cr \Gamma_0 v
\end{pmatrix},\; v\in\cD[\wt S_1]\right\},\;
\wt \cA\left[\begin{pmatrix}v\cr \Gamma_0 v
\end{pmatrix}\right]=\wt S_1[v],
\end{array}
\end{equation}
and $\dom\wt\cA^{1/2}\cap\cH=\{0\}$ holds. If, in addition, the pair
$\{\wt S_0,\wt S_1\}$ satisfies the properties \eqref{nonnegpair},
then
\begin{equation}
\label{EXTR} \inf\limits_{u\in\dom\wt S_0}\wt
\cA\left[\begin{pmatrix}v\cr \Gamma_0 v
\end{pmatrix}-\begin{pmatrix}u\cr 0\end{pmatrix}\right]=0\quad\mbox{for all}\quad v\in\cD[\wt S_1]
\end{equation}
and, moreover,
\[
 \ran\wt\cA^{1/2}\cap\cH=\{0\}.
\]
\end{theorem}
\begin {proof}
It follows from \eqref{A} and \eqref{GG} that
\begin{equation}
\label{OA} \left(\wt \cA\begin{pmatrix}u\cr \Gamma_0 u
\end{pmatrix},\begin{pmatrix}v\cr \Gamma_0 v
\end{pmatrix}\right)_{\wt \sH}=(S^*u,v)-(\Gamma_1u,\Gamma_0 v)=\wt S_1[u,v]\ge 0,\; u, v\in\cL.
\end{equation}
Therefore, the operator $\wt\cA$ is nonnegative and clearly
$S\subset\wt\cA$. Observe that
\begin{equation}
\label{graphS} \graph\wt \cA\cap(\sH\oplus\{0\})^2=\graph S.
\end{equation}
Next it will be proved that $\cR(\wt \cA+I_{\wt\sH})=\wt\sH$. Given
the vectors $h\in\sH$ and $\f\in \cH$ it is shown that the system of
equations
\[
\left\{\begin{array}{l} S^*u+u=h\\
-\Gamma_1u+\Gamma_0 u=\f
\end{array}
\right.
\]
has a unique solution $u\in\cL.$ According to \eqref{decomp11} the
vector $u\in\cL$ has the decomposition $u=u_1+f_{-1}$, where
$u_1\in\dom\wt S_0$ and $f_{-1}\in\wt\sN_{-1}=\cD[\wt
S_1]\ominus_{\wt S_1}\cD[\wt S_0]$. Then $S^*u+u=\wt S_0 u_1+ u_1.$
Since $\wt S_0$ is a nonnegative selfadjoint operator, one obtains
$u_1=(\wt S_0+I_\sH)^{-1}h.$ Then
\[
\f=-\Gamma_1u+\Gamma_0 u=-\Gamma_1 u_1 -\Gamma_1 f_{-1}+\Gamma_0
f_{-1},
\]
i.e.,
\[
-\Gamma_1f_{-1}+\Gamma_0 f_{-1}=\f+\Gamma_1u_1.
\]
It follows from \eqref{gz} and \eqref{Gam10} that for all
$g\in\wt\sN_{-1}$,
\[
(-\Gamma_1 g+\Gamma_0 g,\Gamma_0
g)_\cH=(-\Gamma_1\Gamma_0(-1)\Gamma_0 g+\Gamma_0 g,\Gamma_0
g)_\cH\ge ||\Gamma_0 g||^2_\cH,
\]
and hence the operator $-\Gamma_1\Gamma_0(-1)+I_\cH$ is bounded and
positive definite on $\cH$. It follows that the equation
$-\Gamma_1f_{-1}+\Gamma_0 f_{-1}=\f+\Gamma_1u_1$ has a unique
solution $f_{-1}\in\wt\sN_{-1}.$ Thus, $\cR(\wt \cA+
I_{\wt\sH})=\wt\sH.$ This shows that the operator $\wt \cA$ is
selfadjoint and nonnegative in $\wt\sH$.

Since the form $\wt S_1[u,v]$ is closed in $\sH$, the form
\[
\wt \tau\left[\begin{pmatrix}u\cr \Gamma_0  u
\end{pmatrix},\begin{pmatrix}v \cr \Gamma_0  v
\end{pmatrix}\right]:=\wt S_1[u,v],\; u,v\in\cD[\wt S_1],
\]
is closed in $\wt\sH$ and  by \eqref{OA} the  selfadjoint operator
$\wt \cA$ is  associated with $\wt\tau$ according to the first
representation theorem in \cite{Ka}. This proves \eqref{DomForm}.

The form $\wt S_0[\cdot,\cdot]$ is a closed restriction of the form
$\wt S_1[\cdot,\cdot]$ with $\dom\wt S_0$ being a core of $\cD[\wt
S_0]$. Therefore, under the conditions \eqref{nonnegpair} for $\{\wt
S_0,\wt S_1\}$, the formula \eqref{EXTR} is obtained from Theorem
\ref{new}.

It follows from \eqref{DomForm} that
$$\begin{pmatrix}0\cr h\end{pmatrix}\notin \dom\wt\cA^{1/2},\; h\ne 0.$$
Next assume that
$$\begin{pmatrix}0\cr h\end{pmatrix}\in \ran\wt\cA^{1/2}.$$
Then
\[
\left|\left(\begin{pmatrix}u\cr \Gamma_0 u\end{pmatrix},
\begin{pmatrix}0\cr h\end{pmatrix}\right)\right|^2\le
C\left(\begin{pmatrix}S^*u\cr -\Gamma_1 u\end{pmatrix},
\begin{pmatrix}u\cr \Gamma_0 u\end{pmatrix}\right)
\]
for all $u\in\cL$ and some $C>0.$ Thus
\[
 \left|\left(\Gamma_0u, h\right)_\cH\right|^2\le C\wt S_1[u],\quad
 u\in\cL.
\]
Replacing $u$ by $u-\f$, where $\f\in\dom \wt S_0$, and noting that
$\Gamma_0\f=0$, one obtains
\[
\left|\left(\Gamma_0u, h\right)_\cH\right|^2\le C\wt S_1[u-\f],\quad
 u\in\cL,\; \f\in\dom\wt S_0.
\]
Furthermore, since
\[
 \inf\left\{\wt S_1[u-\f],\;\f\in\cD[\wt S_0]\right\}=0
 \quad\mbox{for all}\quad u\in\cD[\wt S_1];
\]
and $\dom\wt S_0$ is a core of $\cD[\wt S_0]$, one concludes that
\[
\left(\Gamma_0u, h\right)_\cH =0\quad\mbox{for all}\quad u\in\cL.
\]
Now the identity $\Gamma_0\cL=\cH$ implies that $h=0$, i.e.
$\ran\wt\cA^{1/2}\cap\cH=\{0\}.$ The proof is complete.
\end{proof}

Taking into account the definition of a boundary relation and
results established in \cite{DHMS1} we arrive at the following
statement.

\begin{remark}
\label{BR} In the theory of boundary relations \cite{DHMS1} the
operator $\cA$ is called the main transform of the mapping ${\bf
\Gamma}:=\left(\Gamma_0,\Gamma_1\right)$. The selfadjointness of
$\cA$ together with \eqref{graphS} means that $\{\cH,{\bf\Gamma}\}$
is a boundary relation for $S^*$.
\end{remark}
The next statement is a converse to Theorem~\ref{MT}.
\begin{theorem}
Let $S$ be a densely defined symmetric operator in $\sH$ and let
$\{\cH,\Gamma_0, \Gamma_1\}$ be a generalized boundary triplet for
$S^*$ (in the sense of \cite{DM2}) with $\ker \Gamma_i=\wt S_i$,
$i=1,2$, and such that
\begin{enumerate}
\item the main transform $\wt\cA$ in
\eqref{A} is a nonnegative selfadjoint operator,
 \item the closed form associated with $\wt\cA$ is given by
 \[
 \cD[\wt\cA]=\left\{\begin{pmatrix}v\cr
\wt \Gamma_0 v\end{pmatrix}:v\in\cD[\wt S_1]\right\},\;\wt
\cA\left[\begin{pmatrix}v\cr \wt\Gamma_0 v\end{pmatrix}\right]
 =\wt S_1[v],
 \]
where $\wt\Gamma_0$ is a linear operator acting from $\cD[\wt S_1]$
into $\cH$ and extends the mapping $\Gamma_0$, and where $\wt
S_1[u,v]$ stands for the closure of the form $(\wt S_1 u,v)$, $u,v
\in\dom \wt S_1$ (see \eqref{formclos}).
\end{enumerate}
Then $\{\wt S_0,\wt S_1\}$ is a pair of nonnegative selfadjoint
extensions of $S$ which satisfies the conditions \textbf{(a), (b),
(c)}.
\end{theorem}
\begin{proof}
It is first shown that $\wt S_1$ is a selfadjoint operator. Since
$\dom \wt S_1=\ker \Gamma_1$ it is clear from \eqref{A} that $\wt
S_1$ is a nonnegative extension of $S$. Let $\wt S_{1{\rm F}}$ be
the Friedrichs extension of $\wt S_1$ and let $u\in\dom \wt S_{1{\rm
F}}$; cf. \eqref{formclos}. Then by the first representation theorem
\cite{Ka} the equality
$$(\wt S_{1{\rm F}}u,v)=\wt S_1[u,v]$$
is valid for all $v\in\cD[\wt S_1]$. Since $\wt S_{1{\rm
F}}\supseteq\wt S_1\supset S$, we get $\wt S_{1{\rm F}}\subset S^*$
and $\dom \wt S_{1{\rm F}}\subset\cD[\wt S_1].$ Thus,
$$(S^*u,v)=\wt S_1[u,v]=\wt \cA\left[\begin{pmatrix}u\cr\wt \Gamma_0 u\end{pmatrix},
\begin{pmatrix}v\cr\wt \Gamma_0 v\end{pmatrix}\right]$$
for all $v\in\cD[\wt S_1]$. On the other hand

$$(S^*u,v)=\left(\begin{pmatrix}S^*u\cr 0\end{pmatrix},\begin{pmatrix}v\cr \wt\Gamma_0 v\end{pmatrix}
\right)_{\wt\sH}.$$ Making use the first representation theorem
again, we get
\[
u\in\dom \wt S_{1{\rm F}}\Rightarrow\begin{pmatrix}u\cr\wt\Gamma_0
u\end{pmatrix}\in\dom\wt\cA,\;\mbox{and}\quad\wt\cA\begin{pmatrix}u\cr\wt\Gamma_0
u\end{pmatrix}=\begin{pmatrix}S^*u\cr 0 \end{pmatrix}.
\]
Now definition \eqref{A} of the main transform $\wt \cA$ yields the
equality $\Gamma_1 u=0$. This means that $u\in\ker\Gamma_1=\dom \wt
S_1$. Therefore, $\wt S_{1{\rm F}}=\wt S_1$ and thus $\wt S_1$ is
selfadjoint.

It is clear from \eqref{A} and \eqref{DomForm} that the equality
$(\wt S_0u,v)=\wt S_1[u,v]$ holds for all $u,v\in\dom \wt S_0$.
Consequently, $\wt S_0$ is nonnegative and the closed form
corresponding to $\wt S_0$ is a restriction of the closed form
$\wt{S}_1[\cdot,\cdot]$. Therefore, the pair $\{\wt S_0,\wt S_1\}$
satisfies all the conditions in \textbf{(a), (b), (c)}.
\end{proof}


\begin{definition}
\label{weyl} Let $\{\cH,\Gamma_0\}$ be a boundary pair for $\{\wt
S_0,\wt S_1\}$, let $\{\cH,\Gamma_0, \Gamma_1\}$ be the
corresponding boundary triplet, and let $\Gamma_0(z)$ be as in
\eqref{gz}. The operator valued function
\[
M(z):=\Gamma_1\Gamma_0(z),\;z\in\Ext [0,\infty),
\]
is called the Weyl function.
\end{definition}
An application of \eqref{GG} and \eqref{WW} shows that
\[
\left(W(z,\xi)h,e\right)_\cH=z\left(\Gamma_0(z)h,\Gamma_0(\xi)e\right)-
\left(M(z)h,e\right)_\cH,\quad h,e\in\cH.
\]
Hence,
\begin{equation}
\label{QWG} -M(z)=W(z,\xi)-z\Gamma_0^*(\xi)\,\Gamma_0(z).
\end{equation}
Since $W^*(z,\xi)=W(\xi,z)$, this implies that
\[
-M^*(\xi)=W(z,\xi)-\overline\xi\Gamma_0^*(\xi)\,\Gamma_0(z).
\]
Therefore,
\begin{equation}
\label{Nevan}
\frac{M(z)-M^*(\xi)}{z-\overline\xi}=\Gamma_0^*(\xi)\,\Gamma_0(z),
\end{equation}
and
\[
W(z,\xi)=\frac{\overline\xi M(z)-zM^*(\xi)}{z-\overline\xi}.
\]
Next another expression for $M(z)$ is derived by means of the Cayley
transforms $\wt B_k=(I-\wt S_{\rm{K}})(I+\wt S_{\rm{K}})^{-1}$,
$k=0,1$.

Since $\ran\Gamma_0(-1)=\cD[\wt S_1]\ominus_{\wt S_1}\cD[\wt S_0]$,
Proposition \ref{new0} shows that
$$\ran\Gamma_0(-1)=\ran(\wt B_1-\wt B_0)^{1/2}.$$
Therefore, there exists a continuous linear isomorphism $X_0$ from
$\cH$ onto the subspace $\sN_{-1}=\sN$ such that
\[
\Gamma_0(-1)= (\wt B_1-\wt B_0)^{1/2}X_0.
\]

\begin{theorem}
\label{T1} The Weyl function $M(z)$ of the boundary triplet
$\{\cH,\Gamma_0,\Gamma_1\}$ takes the form
\[
\begin{array}{l}
M(z)=-2X^*_0\left(I+(\wt B_1-\wt B_0)^{1/2}\left(\wt B_0-\frac{1-z}{1+z}\, I\right)^{-1}(\wt B_1-\wt B_0)^{1/2}\right)X_0,\\
\qquad\qquad=-2X^*_0\wt  Q_0\left(\cfrac{1-z}{1+z}\right)X_0, \;
z\in \Ext [0,\infty),
\end{array}
\]
where $\wt Q_0$ is defined by \eqref{q1}. If, in particular, the
pair $\{\wt S_0,\wt S_1\}$ satisfies the properties \eqref{spnon},
then $M(z)$ belongs to the class $\sS_{\rm{F}}(\cH)$.
\end{theorem}
\begin{proof}

From \eqref{WW} and \eqref{QWG} one obtains for all $h\in\cH$,
\[
-\left(M(-1)h,h\right)_\cH=\left(W(-1,-1)h,h\right)_\cH+||\Gamma_0(-1)h||^2=||\Gamma_0(-1)h||^2_{\wt
S_1}.
\]
Now the definition of $X_0$ and \eqref{norm} imply
\[
 -\left(M(-1)h,h\right)_\cH = ||(\wt B_1-\wt B_0)^{1/2}X_0h||^2_{\wt
S_1}=2||X_0h||^2,
\]
which leads to
\[
M(-1)=-2 X^*_0X_0.
\]
According to \eqref{Nevan} one has
$M(z)=M(-1)+(z+1)\Gamma_0^*(-1)\Gamma_0(z)$ and using \eqref{gzxi}
and \eqref{preobras} one obtains
\[
\begin{array}{rl}
 M(z)&=M(-1)+(z+1)\Gamma_0^*(-1)(I+\wt S_0)(\wt S_0-z)^{-1}\Gamma_0(-1) \\
 &=-2X^*_0X_0-2X^*_0(\wt B_1-\wt B_0)^{1/2}\left(\wt B_0-\frac{1-z}{1+z}\, I\right)^{-1}(\wt B_1-\wt B_0)^{1/2}X_0 \\
 &=-2X^*_0\left(I+(\wt B_1-\wt B_0)^{1/2}\left(\wt B_0-\frac{1-z}{1+z}\, I\right)^{-1} (\wt B_1-\wt
 B_0)^{1/2}\right)X_0 \\
 &=-2X^*_0\wt Q_0\left(\cfrac{1-z}{1+z}\right)X_0.
\end{array}
\]
Finally, since $X_0$ is a linear isomorphism (homeomorphism) it
follows from Theorem~\ref{adth} that the function $M(z)$ together
with the function $\wt \cQ_0(z)=-2\wt
Q_0\left(\cfrac{1-z}{1+z}\right)$ belongs to the class
$\sS_{\rm{F}}(\cH)$ of Herglotz- Nevanlinna functions.
\end{proof}


By means of \eqref{PROP} and \eqref{inv11} it is seen that
$M^{-1}(z)\in \bL(\cH)$ for all $z\in \Ext [0,\infty)$ and
\[
M^{-1}(z)=-\frac{1}{2}X^{-1}_0\wt
Q_1\left(\frac{1-z}{1+z}\right)X^{*-1}_0,
\]
where $\wt Q_1$ is defined in \eqref{q2}.

In conclusion we mention one more general relation for the Weyl
function $M(z).$ Let $\wt \cA$ be defined by \eqref{A}. Then
\[
P_\cH(\wt\cA-zI)^{-1}\uphar\cH=-(M(z)+zI)^{-1}, \;
z\in\dC\setminus\dR_+.
\]
Indeed, since
\[
\left(\wt\cA-zI\right)\begin{pmatrix}u\cr\Gamma_0u\end{pmatrix}=
\begin{pmatrix}S^*u-zu\cr-\Gamma_1u- z\Gamma_0u\end{pmatrix},
\]
the equality
\[
\left(\wt\cA-zI\right)\begin{pmatrix}u\cr\Gamma_0u\end{pmatrix}=
\begin{pmatrix}0\cr h\end{pmatrix}
\]
holds if $u\in\sN_z\cap\dom\wt\cA=\wt\sN_z$. Hence $u=\Gamma_0(z)e$
for certain $e\in\cH$. Then
$$-\Gamma_1u-z\Gamma_0u=-(M(z)+zI)e=h.$$
Hence $P_\cH(\wt\cA-zI)^{-1}h=e=-(M(z)+zI)^{-1}h.$

\bigskip

\noindent \textbf{Acknowledgements.} The first author thanks the
Department of Mathematics and Statistics of the University of Vaasa
for the hospitality during his visit. The second author is grateful
for the support from the Emil Aaltonen Foundation.

\end{document}